\documentclass[a4paper,11pt]{article}
\usepackage[multidot]{grffile}
\usepackage{hyperref} 
\hypersetup{
    colorlinks=true,
    linkcolor=blue,
    filecolor=magenta,      
    urlcolor=cyan,
    citecolor=red,
    pdfpagemode=FullScreen,
    }

\usepackage[margin=1in]{geometry}
\usepackage{setspace}
\usepackage{floatrow}
\usepackage{enumitem}
\usepackage{hyperref}
\usepackage{array}
\usepackage{subfig, epsfig}
\usepackage{graphicx}
\usepackage{tikz}
\usepackage{color}

\usepackage{supertabular} 

\usepackage{latexsym}
\usepackage{amssymb, amsfonts}
\usepackage{textcomp}
\usepackage{mathrsfs}

\usepackage{amsthm, amsmath}

\usepackage[ruled]{algorithm2e}

\newtheorem{thm}{Theorem}[section]

\theoremstyle{definition}
\newtheorem{definition}[thm]{Definition}

\newcommand{\R}{\mathbb{R}}

\renewcommand{\vec}[1]{\mathbf{#1}}
\newcommand{\uvec}[1]{\mathbf{\hat{#1}}}

\newcommand{\sign}[1]{\operatorname{sign}(#1)}
\newcommand{\cp}{\operatorname{cp}}

\newcommand{\grad}{\nabla}

\newcommand{\be}{\begin{equation}}  
\newcommand{\ee}{\end{equation}}

\newcommand{\reminder}[1]{{#1}}

\begin{document}

\title{A Simple Embedding Method for Scalar Hyperbolic Conservation Laws on Implicit Surfaces}

\author{
Chun Kit Hung\thanks{Department of Mathematics, the Hong Kong University of Science and Technology, Clear Water Bay, Hong Kong. Email: {\bf ckhungab@connect.ust.hk}}
\and
Shingyu Leung\thanks{Department of Mathematics, the Hong Kong University of Science and Technology, Clear Water Bay, Hong Kong. Email: {\bf masyleung@ust.hk}}
}

\date{}

\maketitle

\begin{abstract}
We have developed a new embedding method for solving scalar hyperbolic conservation laws on surfaces. The approach represents the interface implicitly by a signed distance function following the typical level set method and some embedding methods. Instead of solving the equation explicitly on the surface, we introduce a modified partial differential equation in a small neighborhood of the interface. This embedding equation is developed based on a push-forward operator that can extend any tangential flux vectors from the surface to a neighboring level surface. This operator is easy to compute and involves only the level set function and the corresponding Hessian. The resulting solution is constant in the normal direction of the interface. To demonstrate the accuracy and effectiveness of our method, we provide some two- and three-dimensional examples.
\end{abstract}

\section{Introduction}

Let $\Gamma$ be a smooth, oriented, co-dimension $1$ (compact, Riemannian) manifold without boundary in $\R^d$ with $d=2$ or 3. This paper considers solving a scalar conservation law posed on $\Gamma$ given by
\begin{align} \label{eq:SurfaceConservationLaw}
\begin{split}
&    u_t + \grad_{\Gamma} \cdot \Vec{F}(\vec{x}, u) = 0,\\
&    u(\vec{x}, 0) = u_0(\vec{x}), 
\end{split}
\end{align} 
where the initial condition $u_0$ is bounded, and $\vec{F}(\cdot, u)$ is a smooth surface flux function that satisfies the geometry-compatible condition:
$$
\grad_\Gamma \cdot \vec{F}(\vec{x}, \Bar{u}) = 0, \quad \vec{x} \in \Gamma, \Bar{u} \in \R.
$$
As discussed in \cite{ben2007well}, with the geometry-compatible condition and the boundedness of $u_0$, the surface conservation law \eqref{eq:SurfaceConservationLaw} has a unique (entropy) solution and is thus well-posed. This partial differential equation (PDE) on surfaces arises in many fields with real-life applications \cite{haltiner1971numerical, ben2007well, ben2009hyperbolic}. Therefore, it is essential to develop efficient numerical algorithms for the problem. We can categorize these methods into mainly two classes.

One approach for solving these equations is based on surface parametrization. We can rewrite the surface PDEs using the corresponding coordinate system and obtain an explicit representation of the differential operator. Numerically, we apply the same techniques next for PDEs defined in the Euclidean space $\R^d$. Although the problem could be reduced to a well-studied one, constructing such a parametrization is complex and impractical for complicated surfaces. One could also triangulate the surface and then locally approximate the differential operator on the resulting triangles. This approach has been widely used for various classes of PDEs \cite{dziuk1988finite,dziuk2013finite,flyer2009radial,olshanskii2010finite}.

Instead of having an explicit representation of the surface, i.e., a specific discretization of the surface, another popular approach is to solve the PDEs on implicit surfaces. These embedding methods represent the surface $\Gamma \subset \R^n$ implicitly using a signed-distance level set function $\phi: \R^n \to \R$ \cite{bertalmio2001variational,sethian2008solving, adalsteinsson2003transport}. That is, $\Gamma = \phi^{-1}(0)$ with $\|\nabla \phi\|=1$. With the unit normal to $\Gamma$ defined as $\uvec{n} = \grad \phi$, one introduces the projection operator as $\mbox{\underline{P}} := I - \uvec{n} \otimes \uvec{n}$, so that the surface gradient of $u$ can be expressed as $\grad_{\Gamma}u = \mbox{\underline{P}}\grad \Tilde{u} = \grad \Tilde{u} - (\uvec{n} \cdot \grad{\Tilde{u}})\uvec{n}$, where $\Tilde{u}: \R^n \to \R^n$ is the embedding of $u: {\Gamma} \to \R$, with $\Tilde{u}|_{\Gamma} = u$. Similarly, the surface Laplacian (or the so-called Laplace-Beltrami operator) of $u$ can be rewritten as $\Delta_{\Gamma} u = \mbox{\underline{P}}\grad \cdot (\mbox{\underline{P}}\grad \Tilde{u}) =\grad \cdot (\mbox{\underline{P}}\grad \Tilde{u})$. The resulting embedded PDE can be handled easily since all differential operators on surfaces are replaced by standard Cartesian operators, which can be discretized using typical finite difference approaches. To save some computational power, one might simply consider the solution $\Tilde{u}$ in a small computational tube containing $\Gamma$. A common choice is to use a tubular neighbourhood\footnote{We follow the terminology from differential geometry and use the term \textit{tubular neighborhood} here since it covers the definition from arbitrary dimensions. For some special cases, we might use \textit{annulus}, \textit{tube} or \textit{shell} to explain the setting better.} of $\Gamma$, given by ${\Gamma}^{R} := \{\vec{x} \in \R^n: |\phi| \leq R\}$, for some tube radius $R > 0$. When restricted to $\Gamma$, the solution to the embedding PDE gives the solution to the surface PDE.

The method has been further developed in \cite{xu2003eulerian} for solving the convection-diffusion equation on moving interfaces and in \cite{greer2006fourth} for solving a fourth-order PDE on surfaces. To obtain an accurate numerical solution, both works have required re-extending the surface data in $\Gamma^{R}$ to guarantee a constant normal derivative of $\Tilde{u}$ for every time step. Note that one also needs an additional boundary condition on $\partial \Gamma^{R}$ in the numerical implementation. Since we only require that the embedding PDE agrees with the surface PDE on the zero level set, we have some flexibility in how we introduce the embedding PDE. To improve the accuracy of these embedding methods, the work in \cite{greer2006improvement} introduced a modified projection operator $\Tilde{\mbox{\underline{P}}} = (I - \phi K)^{-1}\mbox{\underline{P}}$, where $\phi$ is the signed distance level set function such that $\Gamma=\phi^{-1}(0)$ and $\| \nabla \phi \|=1$, and $K$ is the curvature tensor of each level set given by the negative of the level set Hessian. If the initial surface data $\Tilde{u}_0$ is constant along the normal directions, it can be shown that the numerical solution $\Tilde{u}$ to the embedding PDE $\Tilde{u}_t = \mathcal{L}(\Tilde{u}, \Tilde{\mbox{\underline{P}}}\grad \Tilde{u}, \grad \cdot (\Tilde{\mbox{\underline{P}}}\grad \Tilde{u}))$ is always constant along the normal directions. Mathematically, we have $\grad \Tilde{u}(t) \cdot \grad\phi = 0$ for $t>0$ if the initial condition satisfies $\grad \Tilde{u}_0 \cdot \grad\phi = 0$.

Related to this level set approach, the closest-point method (CPM) has been developed \cite{ruuth2008simple, macdonald2008level} to solve time-dependent PDEs on implicit surfaces. With the closest-point function $P_{\Gamma} : \R^n \to S$, which maps a grid point $\vec{x}$ to the closest-point $P_{\Gamma}(\vec{x}) = \arg \min_{\vec{z} \in \Gamma} \|\vec{x} - \vec{z}\|$, one can rewrite the surface gradient as $\grad_{\Gamma}u(\vec{x}) = \grad u(P_{\Gamma}(\vec{x}))$ on $\Gamma$, so that the surface gradient can be computed using the Euclidean gradient.
Typical embedding methods discussed above concentrate on elliptic or parabolic PDEs. Since the solution to hyperbolic PDEs might develop discontinuities, most arguments in the above embedding method might break down when the gradient $\grad u$ no longer exists. Therefore, despite many important applications, there has not been much attention paid to hyperbolic PDEs on surfaces. One approach developed in \cite{macdonald2008level} is for solving the level set equation on surfaces. Another discussion is on the surface Eikonal equations in \cite{wong2016fast}.


This paper develops a new embedding method for solving scalar hyperbolic conservation laws on surfaces. \reminder{Drawing inspiration from differential geometry,} we introduce the push-forward operator to extend the surface flux function. The approach can provide a flux function in a small interface neighborhood that is still tangential to the interface. Therefore, this push-forward operator modifies the surface PDE and gives a PDE defined in a computational tube. Following the idea in other embedding methods, we can apply any well-developed numerical method to solve the modified PDE in the Cartesian mesh. More importantly, we can show analytically that the solution to the modified equation is constant along the normal directions of the interface. This property is essential \reminder{for maintaining} numerical accuracy in the finite difference scheme for the differential operator and the interpolation step when extracting the solution from the mesh to the surface.

The rest of the paper is organized as follows. We introduce the push-forward operator from differential geometry and propose a simple numerical method that can give high-order accurate numerical schemes in Section \ref{Sec:OurProposedMethod}. A careful comparison of our proposed approach with \reminder{several} existing embedding methods \reminder{is given in Section \ref{Sec:Comparison}.} \reminder{Implementation details are discussed in Section \ref{Sec:NumericalImplementaion}, while Section \ref{Sec:Examples} offers two- and three-dimensional examples to illustrate the accuracy of our numerical approach.} Section \ref{Sec:Conclusion} gives a conclusion and suggests some future work.


\section{Our Proposed Method}
\label{Sec:OurProposedMethod}

This section first introduces some tools from differential geometry, and then we propose a simple approach to embed the conservation laws in a small tubular neighborhood $\Gamma^{R}$. We prove that the solution to the embedding conservation law \eqref{eq:EmbeddingConservationLaw} is constant along the normal directions of $\Gamma$ during time evolution.

\subsection{The closest-point Function as a Diffeomorphism}
\reminder{
Let $\Gamma \subset \R^d$ be a manifold as described above, we can define the closest-point function of $\Gamma$, by 
\[
P_{\Gamma}(\vec{x}) := \arg \min_{\vec{y} \in \Gamma} \|\vec{x} - \vec{ y}\|   
\]
for $\vec{x}$ near $\Gamma$. In particular, since $\Gamma = \phi^{-1}(0)$ for its signed distance function $\phi$, the closest-point function can be computed explicitly by 
\[
P_{\Gamma}(\vec{x}) = \vec{x} - \phi(\vec{x})\grad \phi(\vec{x}),
\] 
where $\vec{x} \in \Gamma^{R} = \{|\phi(\vec{x})| < R\} \subset \R^d$ is the tubular neighbourhood of $\Gamma$.} If the given interface has a form other than the signed distance representation, we could have the following preprocessing steps before our proposed algorithm. 
\reminder{If the level set is not a signed distance function, one can follow the reconstruction procedure described in \cite{houlizoshzha97,say14} to obtain a high-order approximation to the closest-point.}
If an explicit parametrization of the interface is known (for example, $\vec{f}(\vec{s})$ for some parameterization $\vec{s}$), we can determine the shortest distance and the corresponding nearest point by minimizing the function $d(\vec{s};\vec{x})=\frac{1}{2}\| \vec{f}(\vec{s})-\vec{x}\|^2$.

Based on this closest-point representation, we can compute the surface gradient and surface divergence intrinsic to $\Gamma$ using their Cartesian counterpart \cite{ruuth2008simple}, i.e. 
\begin{thm} \label{thm: Equivalent of Operators}
Let $\Gamma$ be a smooth, closed surface embedded in $\R^d$, given by the zero level set of its signed distance function $\phi$.
Let $u$ be a smooth function whose value is constant along the normal directions of $\Gamma$, then
    $\grad_{\Gamma} u(\vec{x}) = \grad u(P_{\Gamma}(\vec{x}))$
     for $\vec{x} \in \Gamma$.     
Let $\vec{V}$ be a vector field which is tangent to all level sets $\Gamma_{h} = \phi^{-1}(h)$ for small $h$, then
     $
     \grad_{\Gamma} \cdot \vec{V}(\vec{x}) = \grad \cdot \vec{V}(P_{\Gamma}(\vec{x})), 
     $
     for $\vec{x} \in \Gamma$.
\end{thm}
The closest-point function suggests a natural way to extend a surface quantity $u:\Gamma \to \R$ to its tubular neighborhood $\Gamma^{R}$, given by $u(P_{\Gamma}(\cdot))$. This extension is constant along the normal directions of $\Gamma$, hence we have $\grad_{\Gamma} u(\vec{x}) = \grad u(P_\Gamma(\vec{x}))$ on $\Gamma$. However, since the closest-points of computational grids are generally not grid points, the closest-point extension requires interpolation of closest-point values during the evolution. If the surface quantity $u$ develops discontinuities, obtaining a high-order interpolation procedure would be difficult.

\reminder{Next}, we provide a slightly different interpretation of the closest-point function $P_\Gamma$. Consider $\Gamma_h \subset \R^d$, representing the $h$-level set of $\phi$, where $h>0$ is small such that $P_\Gamma$ is well-defined. Define $P_h: \Gamma_h \to \Gamma$ by $P_h(\vec{x}) = \vec{x} - h\grad\phi(\vec{x})$, which is the restriction of $P_\Gamma$ from the tubular neighborhood $\Gamma^{R}$ to the $h$-level set $\Gamma_h$. Notice that $P_h$ is a diffeomorphism between the $h$-level set $\Gamma_h$ and the interface $\Gamma$. This is because $\phi$ is smooth, and its inverse $P_h^{-1}: \Gamma \to \Gamma_h$ is given by $P_h^{-1}(\vec{y}) = \vec{y} + h\grad\phi(\vec{y})$, which is also smooth. Given a smooth function $u$ on $\Gamma$, the most natural way to generate a function on the $h$-level set $\Gamma_h$ is to \textbf{pull-back} $u$ by the operator $P_h$. This pull-back operator of $P_h$ is given by
$$
P_h^{*}u = (u \circ P_h)(\vec{x}) = u(P_h(\vec{x})), \quad \vec{x} \in \Gamma_h.
$$
This gives the normal extension in the closest-point method when restricted to $\Gamma_h$.


\subsection{Our Proposed Method}

The key idea behind our proposed method is that each level set $\Gamma_h$ is diffeomorphic to $\Gamma$ through the closest-point function, and we should be able to pose a PDE on $\Gamma_h$ describing the same physical phenomena. A natural way to pass a tangent vector field between $\Gamma$ and $\Gamma_h$ is called the \textbf{push-forward}. 
\begin{definition}[Push-Forward]
Given a tangent vector field $\vec{F}_{h}$ on $\Gamma_h$, the push-forward of $\vec{F}_{h}$ by $P_h$ at $\vec{x} \in \Gamma_h$ is a linear map given by
\[
(P_h)_*\vec{F}_{h}(\vec{x}) = (P_h \circ \alpha)'(0),
\]
where $\alpha: (-1,1) \to \Gamma_h$ with $\alpha(0) = \vec{x}$ and $\alpha'(0) = \vec{F}_{h}(\vec{x})$. 
\end{definition}
Geometrically speaking, the push-forward $(P_h)_*\vec{F}_{h}(\vec{x})$ is the initial velocity of the image curve $(P_h \circ \alpha)(t)$ on $\Gamma$, with the base point of $(P_h)_*\vec{F}_{h}(\vec{x})$ given by $P_h(\vec{x})$, as shown in Figure \ref{Fig:Pushforward}. It is also known that the push-forward operator is independent of the local coordinates and the choice of $\alpha$ \cite{gray2017modern}.

\begin{figure}[!ht]
    \centering
    \includegraphics[width=0.3\textwidth]{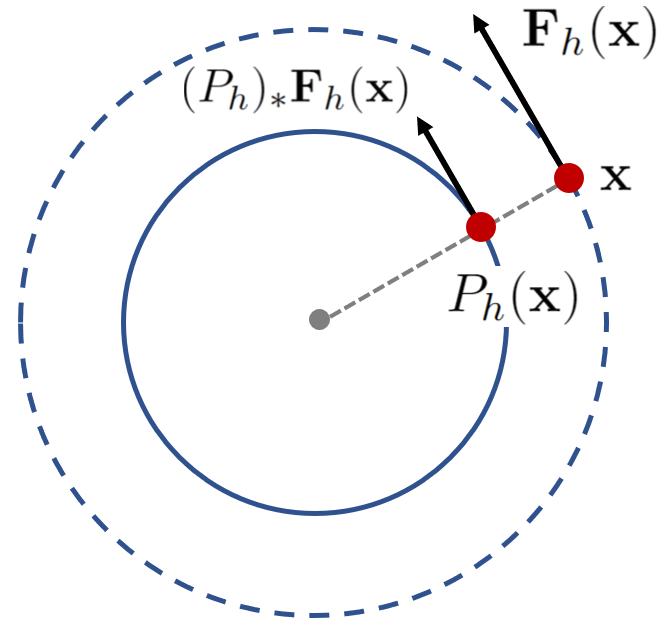}
    \caption{ The push-forward operator.}
    \label{Fig:Pushforward}
\end{figure}

In our application, we only have the surface data on $\Gamma$, and we may consider the push-forward of $\vec{F}$ by $P^{-1}_h$ instead. Here, we give an explicit formula for computing $P^{-1}_h$. This formula does not rely on local coordinates, which perfectly fits our level set representation of $\Gamma$.

\begin{thm} \label{thm: Pushforward Matrix}
Let $\vec{F}$ be a tangent vector field defined on $\Gamma$, where $\Gamma$ is the surface described above. For any $\vec{x} \in \Gamma_h$, the push-forward of $\vec{F}$ by $P_h^{-1}$ is given by 
$$
\vec{F}_h(\vec{x}) = [(P_{h}^{-1})_{*}\vec{F}](\vec{x})= \left[I - hH(\vec{x})\right]^{-1}\vec{F}(P_h(\vec{x})),
$$
where $H$ is the Hessian matrix of $\phi$.
\end{thm}

\begin{proof}
Fix $\vec{x} \in \Gamma_h$, let $\vec{v}(\vec{x})$ be any tangent vector of $\Gamma_h$ at $\vec{x}$, we first compute $(P_h)_*\vec{v}(\vec{x})$. Let $\alpha: (-1,1) \to \Gamma_h$ be a smooth curve on $\Gamma_h$ satisfying $\alpha(0) = \vec{x}$ and $\alpha'(0) = \vec{v}(\vec{x})$, then
\begin{align*}
(P_h)_*\vec{v}(\vec{x}) 
&= (P_h \circ \alpha)'(0) = \frac{d}{dt} \big[\alpha(t) - h\grad \phi(\alpha(t))\big]\Big|_{t=0} \\
&= \big[\alpha'(t) - hH(\alpha(t)) \alpha'(t)\big]\Big|_{t=0} = \left[I - hH(\alpha(0))\right]\alpha'(0) \\
&= \left[I - hH(\vec{x})\right] \vec{v}(\vec{x}) \, .
\end{align*}
Therefore, $(P_{h})_*$ at $\vec{x}$ has the matrix representation $\left[I - hH(\vec{x})\right]$, which implies that $(P^{-1}_{h})_*$ at $P_{h}(\vec{x})$ can be expressed as $\left[I - hH(\vec{x})\right]^{-1}$. 
%
\end{proof}

We apply the push-forward operator to solve conservation laws on surfaces. Let $\vec{F}(\vec{x}, u)$ be the surface flux function on $\Gamma$ described before.
We have a conservation law posed on $\Gamma_h$:
\begin{align*}
\begin{split}
&    \Tilde{u}_t + \grad_{\Gamma_h} \cdot \vec{F}_{h}(\vec{x}, \Tilde{u}) = 0, \\
&    \Tilde{u}(\vec{x}, 0) = u_0(P_{h}(\vec{x})).
\end{split}
\end{align*}
Therefore, introducing the embedding flux $\Tilde{\vec{F}}$ in $\Gamma^{R}$ as
$$
    \Tilde{\vec{F}}(\vec{x}, u) = \left[I - \phi(\vec{x})H(\vec{x})\right]^{-1} \vec{F}(P_{\Gamma}(\vec{x}), u) \, ,
$$
we obtain the embedding problem in the whole tubular neighborhood:
\begin{align} \label{eq:EmbeddingConservationLaw}
\begin{split}
&    \Tilde{u}_t + \grad \cdot \Tilde{\vec{F}}(\vec{x}, \Tilde{u}) = 0, \\
&    \Tilde{u}(\vec{x}, 0) = u_0(P_{\Gamma}(\vec{x})).
\end{split}
\end{align}

Since the push-forward operator does not change the smoothness of the flux function, the uniqueness of \eqref{eq:EmbeddingConservationLaw} follows the original problem on $\Gamma$ \eqref{eq:SurfaceConservationLaw}. For the rest of this section, we prove that, if $u(\vec{x},t)$ is the solution to the surface conservation law \eqref{eq:SurfaceConservationLaw}, then $u(P_{\Gamma}(\vec{x}),t)$ is the solution to equation \eqref{eq:EmbeddingConservationLaw}. This property means that the solution to equation \eqref{eq:EmbeddingConservationLaw} is always constant along the normal directions of $\Gamma$. This property enables us to replace the surface divergence with the simple Cartesian divergence in the computation. Since our method does not require additional interpolations, our approach is particularly attractive for hyperbolic problems when the solution might contain shocks and discontinuities.

\begin{thm}
The solution $\Tilde{u}$ to equation \eqref{eq:EmbeddingConservationLaw} is constant along the normal directions of $\Gamma$ for $t>0$.
\end{thm}

\begin{proof}
Let $\vec{x}_0 \in\Gamma$, and denote the characteristic curve of \eqref{eq:EmbeddingConservationLaw} as $\vec{x}(t)$, where $\vec{x}(0) = \vec{x}_0$. Also, let $\vec{x}_h(t)$ be the characteristic curve starts at $P^{-1}_h(\vec{x}_0)$ on $\Gamma_h$. First, we can write down the characteristic system for $\vec{x}(t)$:
\begin{align*}
   \vec{x}'(t) 
   &= \Tilde{\vec{F}}_{u}(\vec{x}(t), \Tilde{u}(\vec{x}(t), t)) = \vec{F}_{u}(\vec{x}(t), \Tilde{u}(\vec{x}(t), t)), \\
    \frac{d\Tilde{u}}{dt}(t) &= - (\grad \cdot \Tilde{\vec{F}})(\vec{x}(t), \Tilde{u}(\vec{x}(t), t)) = 0,
\end{align*}
where we use the geometry-compatible condition on $\Gamma$ in the second equation. Similarly, we can also write it down for $\vec{x}_{h}(t)$:
\begin{align*}
   \vec{x}_h'(t) &= \Tilde{\vec{F}}_{u}(\vec{x}_h(t), \Tilde{u}(\vec{x}_h(t), t)), \\
    \frac{d\Tilde{u}}{dt}(t) &= - (\grad \cdot \Tilde{\vec{F}})(\vec{x}_h(t), \Tilde{u}(\vec{x}_h(t), t)) = 0.
\end{align*}
The second equation for $\Tilde{u}$ equals to zero (using Lemma 1.1 in \cite{d1969divergence}). We can see that, the geometry-compatible condition guarantees that $\Tilde{u}$ is constant along the characteristic curves $\vec{x}(t)$ and $\vec{x}_h(t)$. 

Finally, we need to show $\vec{x}_h(t) = P^{-1}_h(\vec{x}(t))$, i.e. the image curve of $\vec{x}(t)$ on $\Gamma_h$, is given by $\vec{x}_h(t)$. Consider
\begin{align*}
\frac{d}{dt}P^{-1}_{h}(\vec{x}(t)) 
&= \Big[I - hH\big(P^{-1}_h(\vec{x}(t))\big)\Big]^{-1}\vec{x}'(t) \\
&= \Big[I - hH\big(P^{-1}_h(\vec{x}(t))\big)\Big]^{-1} \vec{F}_{u}(\vec{x}(t), \Tilde{u}(\vec{x}(t), t)),
\end{align*}
and
$$
\vec{x}_h'(t) = \Tilde{\vec{F}}_{u}(\vec{x}_h(t), \Tilde{u}(\vec{x}_h(t), t)) = \Big[I - hH\big(\vec{x}_h(t)\big)\Big]^{-1} \vec{F}_u(P_{\Gamma}(\vec{x}_h(t)), \Tilde{u}(\vec{x}_h(t), t)).
$$
Since $\vec{x}(0) = P_{h}(\vec{x}_h(0))$, we have $\frac{d}{dt}P^{-1}_{h}(\vec{x}(t))\big|_{t=0} = \vec{x}'_h(0)$. For $t > 0$, these two curves have to be identical. 

Since the initial condition as defined in \eqref{eq:EmbeddingConservationLaw} is constant along the normal directions, we can then conclude that the solution $\Tilde{u}$ is also constant along the normal directions of $\Gamma$ for all $h$ as long as the closest-point function $P_{\Gamma}$ is well-defined. 
\end{proof}

\subsection{Eigenvalues of the Push-Forward Operator}
\label{SubSec:PushforwardEigenvalue}
There is a recent embedding method for solving Hamilton-Jacobi equations on surfaces, where the unique viscosity solution is given by the normal extension of the solution to the original surface problem \cite{martin2020equivalent}. The key component of their method is the singular values of the Jacobian matrix of $P_\Gamma$, given by $\sigma_{i} = 1 - \phi(\vec{x})\kappa_{i}(\vec{x})$, where $\kappa_i$ are the principal curvatures of $\Gamma_{\phi(\vec{x})}$ at $\vec{x}$. 

We found that when restricted to $\Gamma_h$, the eigenvalues of the push-forward matrix $\left(P_h^{-1}\right)_{*}$ we discuss here resembles $\sigma_i^{-1}$ as mentioned in \cite{martin2020equivalent}. Here, we concentrate only on the surface case where $d = 3$, but the discussion in the two-dimensional curve case is similar and, therefore, omitted.

\begin{thm} \label{thm: Eigenvalues}
Let $\left(P_h^{-1}\right)_{*}$ be the push-forward matrix. The eigenvalues of the operator are given by $\lambda_1 = (1 - h\kappa_1)^{-1}$, $\lambda_2 = (1 - h\kappa_2)^{-1}$ and $1$, with the associated eigenvectors given by the corresponding principal direction vectors $\vec{p}_1, \vec{p}_2$ of the level surface $\phi(\vec{x})$ and the normal vector $\grad \phi(\vec{x})$.
\end{thm}

\begin{proof}
It is known that for the signed distance function $\phi$, the curvature tensor is given by its Hessian matrix \cite{lehmann2012notes}. Differentiating both sides of $\grad \phi \cdot \grad \phi = 1$, we have $H(\vec{x}) \grad \phi(\vec{x}) = \vec{0}$ which implies that $\grad \phi(\vec{x})$ is an eigenvector of $H(\vec{x})$ associated to the eigenvalue $0$. Another two eigenvectors are given by the principal direction vectors $\vec{p}_1, \vec{p}_2$ are the eigenvectors associated to $\lambda_1$ and $\lambda_2$. Then, the eigenvalues of $I - hH(\vec{x})$ is clearly given by $1-h\lambda_{1}, 1-h\lambda_{2}$ and $1$.
\end{proof}


\subsection{A Simple Example}
\label{SubSec:Example_Circle}

In this simple example, we provide some details to demonstrate the effectiveness of our proposed approach. We consider the unit circle $\Gamma$ defined implicitly by the zero level set of the signed distance function $\phi(x,y) = \sqrt{x^2 + y^2} - 1$. The Hessian matrix of $\phi$ is given by
$$
H(x,y) = 
\frac{1}{r^3}
\begin{bmatrix}
 y^2 & -xy\\
-xy & x^2 \\
\end{bmatrix}
$$
where $r=\sqrt{x^2+y^2}$. Therefore, the eigenvalues of $H(x,y)$ are given by $0$ and $r^{-1}$ with the corresponding eigenvectors given by $(x,y)$ and $(-y,x)$ respectively, which are nothing but simply the normal and tangent vectors to the circle.

\begin{figure}[!ht]
\centering
\includegraphics[width=0.5\textwidth]{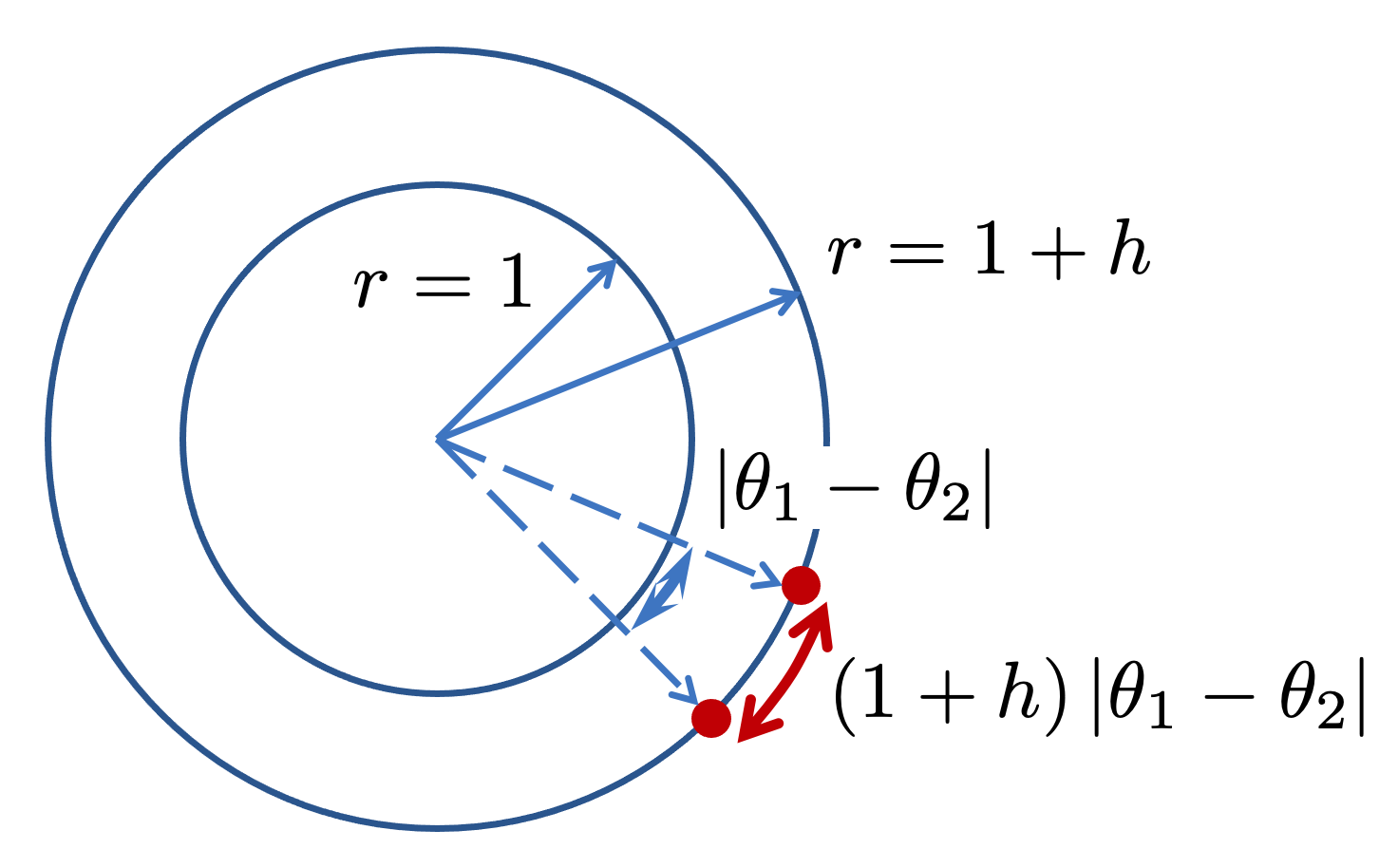}
\caption{(Section \ref{SubSec:Example_Circle}) 
Geodesic distances between two points and their projection.
}
\label{Fig:Example_Circle}
\end{figure}

\reminder{Then}, consider the computation tube $\Gamma^{R} = \{ |\phi| < R\}$ is an annulus with radius $R > 0$ and two points $(r,\theta)=(1+h,\theta_1)$ and $(1+h, \theta_2) \in \Gamma_h \subset \Gamma^R$ written in the polar coordinates, on the $h$-level set of $\phi$. As shown in Figure \ref{Fig:Example_Circle}, the geodesic distance between these two points on $\Gamma_h$ is given by $(1+h)|\theta_1-\theta_2|$. Since the geodesic distance between the projections of these two points on $\Gamma$ is $|\theta_1-\theta_2|$, we should adjust the propagation velocity on the $h$-level set by a factor of $(1+h)$ to account for the difference in the distance travelled. 

According to the discussion in the last section, we see that the eigenvalues of our push-forward matrix $\left[I - h H(x,y)\right]^{-1}$ are given by $1$ and $(1 - h\kappa)^{-1}= 1+h$ where $\kappa$ is the curvature of the circle with radius $1+h$. This term matches exactly the required factor to align the propagation speed for different level sets.


\section{Comparison with Other Embedding Approaches}
\label{Sec:Comparison}

This section \reminder{provides} a detailed comparison and \reminder{demonstrates} the differences between our proposed method and some other existing embedding methods.

\subsection{The Modified Projection Operator \cite{greer2006improvement}}
\label{Sec:RelationshipModifiedProjection}

The work \cite{greer2006improvement} introduced a modified projection operator $\Tilde{\mbox{\underline{P}}} = (I - \phi K)^{-1}\mbox{\underline{P}}$, where $\phi$ is a signed distance representation of the level set function, and $K$ is the curvature tensor of each level set given by the negative of the level set Hessian. This correction to the projection operator shares some similarities with our proposed approach. Both expressions involve the level set function and its Hessian, and both can produce a solution that is constant along the normal directions of the level sets.

One minor difference is that there is a negative-sign mismatch in the operator. The modified projection operator contains the operator $(I-\phi K)^{-1} = (I + \phi H)^{-1}$, which looks different from our operator $(I-\phi H)^{-1}$ in the embedding equation. We believe this difference might come from how one defines the sign of the curvature.

The significant difference between these two approaches is the fundamental concept of how we construct the embedding equation. The work of \cite{greer2006improvement} followed \cite{bertalmio2001variational} using the idea of projection, which does not appear and is not needed in our formulation. Such an operator has to be applied to all surface gradient operators in the PDE and enforces that information propagates only on each level set.

To demonstrate the differences from our proposed approach, we consider applying the techniques from \cite{greer2006improvement} to the advection equation and a general hyperbolic conservation law, \reminder{and write down explicitly the corresponding equations}.

\begin{itemize}
\item As demonstrated in \cite{greer2006improvement}, one obtains the embedding equation $\Tilde{u}_t + \vec{V} \cdot \Tilde{\mbox{\underline{P}}} \grad \Tilde{u}=0$ when using the projection technique for advection equations. It leads to
$$
\Tilde{u}_t + \Tilde{\mbox{\underline{P}}}^T \vec{V} \cdot \grad \Tilde{u} = 0, ,
$$
which is different from our formulation $\Tilde{u}_t + \left[I - \phi H\right]^{-1}\vec{V}\cdot \grad \Tilde{u} = 0$. The approach in \cite{greer2006improvement} requires an extra projection in front of the push-forward operator, which is missing in our formulation.
\item The work in \cite{greer2006improvement} \reminder{does} not consider any general hyperbolic conservation laws. But following the basic principle, \reminder{we} might consider the embedding equation
\begin{equation}
\Tilde{u}t + \Tilde{\mbox{\underline{P}}} \grad \cdot \vec{F}(P_{\Gamma}(\vec{x}), \Tilde{u}) =0 \, ,
\label{eq:GreerConservationLaws}
\end{equation}
Since the push-forward operator is a function of space, this equation does not match with our formulation (\ref{eq:EmbeddingConservationLaw}). Equation (\ref{eq:GreerConservationLaws}) has an extra projection operator in front of the divergence operator, which makes it in the non-conservative form. When we convert our approach to a similar non-conservative form, our corresponding equation gives a source term on the right-hand side of the equation. Therefore, these two equations are indeed different from each other.
\end{itemize}

\subsection{The \reminder{Closest}-Point Method \cite{macdonald2008level}}
The closest-point method involves substituting standard surface gradients and divergence operators with typical Cartesian derivatives in the embedding space. This substitution comes at the cost of replacing the spatial variable $\vec{x}$ with the closest-point evaluation $\cp(\vec{x})$, i.e., $P_{\Gamma}(\vec{x})$. As a result, the approximation of the derivative in a local neighborhood involves an extra step of high-order, high-dimensional interpolation, leading to a nonlocal evolution equation. For instance, applying the closest-point method to a surface PDE $u_t=F[t,\vec{x},u,\nabla_{\Gamma} u, \nabla_{\Gamma} \cdot ( \nabla_{\Gamma} u/\|\nabla_{\Gamma} u\|) ]$ with the initial condition $u(\vec{x},0)=u_0(\vec{x})$ yields
\begin{eqnarray*}
&& \Tilde{u}_t =F \left[t,\cp(\vec{x}),\Tilde{u}(\cp(\vec{x})),\nabla \Tilde{u}(\cp(\vec{x})), \nabla \cdot \left( \frac{\nabla \Tilde{u}(\cp(\vec{x}))}{\|\nabla \Tilde{u}(\cp(\vec{x}))\|} \right) \right]\\
&& \Tilde{u}(\vec{x},0) = {u}_0(\cp(\vec{x})) \, . 
\end{eqnarray*}
Therefore, when we have conservation laws, the method gives
\begin{eqnarray*}
&& \Tilde{u}_t + \grad \cdot \vec{\Tilde{F}}(\cp(\vec{x}), \Tilde{u}(\cp(\vec{x}))) = 0 \\
&& \Tilde{u}(\vec{x},0) = {u}_0(\cp(\vec{x})) \, . 
\end{eqnarray*}
Although this equation looks similar to our formulation (\ref{eq:EmbeddingConservationLaw}), these two equations are different. The equation from the closest-point method is an evolution equation when one needs to incorporate an interpolation step to obtain $\Tilde{u}(\cp(\vec{x}))$ for $\vec{x}$ in a small neighborhood of the grid point $\vec{x}_i$ of interest. In contrast, we are developing an entire PDE-based method where we can easily approximate all partial derivatives using finite-difference methods.

\subsection{The Straightforward Embedding Method \cite{wong2016fast}}
If we apply the simple embedding idea in \cite{wong2016fast} to the surface scalar conservation law \eqref{eq:SurfaceConservationLaw}, we obtain
\begin{eqnarray*}
&& \Tilde{u}t + \grad \cdot \vec{\Tilde{F}}(\vec{x}, \Tilde{u}) = 0 , \\
&& \Tilde{u}(\vec{x},0) = u_{0}(P_{\Gamma}(\vec{x})) \, ,
\end{eqnarray*}
where $\vec{\Tilde{F}}(\vec{x}, \Tilde{u}) := \vec{F}(P_{\Gamma}{(\vec{x})}, \Tilde{u})$. The function $P_{\Gamma}(\vec{x})$ is the closest-point function of $\vec{x}$ onto the interface $\Gamma$. It can be shown that with such a simple embedding idea, the embedded flux function $\vec{\Tilde{F}}$ does not generate any flux across different level surfaces. Therefore, this straightforward approach also produces an embedding equation consistent with the surface conservation law. Although this equation looks similar to our formulation (\ref{eq:EmbeddingConservationLaw}), how we modify the flux function differs. This straightforward approach does not incorporate an extra operator $\left[I - \phi H\right]^{-1}$ in the flux function. Because of the lack of this extra factor, information on different level surfaces propagates at different speeds depending on the curvature of the corresponding level surface. The characteristics propagate faster on level sets with more significant curvature and slower on level sets with smaller curvature. Although the solution restricted to the zero level set gives the required solution to the surface problem, solutions on different level set surfaces generally differ from those obtained by our formulation (\ref{eq:EmbeddingConservationLaw}).

\section{Numerical Implementations}
\label{Sec:NumericalImplementaion}

\reminder{In this section}, we summarize the proposed method and discuss the implementation details. \reminder{Our computational domain is a uniform Cartesian grid of size $\Delta x$ with an embedded tubular region where computations are carried out. The current implementation generates the push-forward matrix for every grid point in the embedded domain, despite a significant proportion of these grid points being unused.} To improve memory efficiency, one could use special data structures like a list to store all mesh points within the computational tube efficiently. This approach could significantly reduce the memory requirement for high-dimensional computations. However, further investigation is needed to explore this possibility in the future.

Our proposed computational algorithm is straightforward. The method requires some pre-processing steps, including a step to construct the push-forward matrix and an interpolation step to extend the initial condition on the surface to the computational tube. Then we can solve the modified conservation law in the embedded domain as in a typical problem in the Cartesian space.

\vspace{0.2cm}
\noindent
\textsf{Step 1}: Compute the closest-points for the grid points. For each grid point $\vec{x}$, we determine the corresponding closest-point on $\Gamma$ using $P_{\Gamma}(\vec{x}) = \vec{x} - \phi(\vec{x})\grad \phi(\vec{x})$, with $\phi$ as the signed distance function representation of $\Gamma$.

\vspace{0.2cm}
\noindent
\textsf{Step 2}: Determine the computational tube $\Gamma^R$. Following other embedding methods for solving PDE\reminder{s} on surfaces, we require the tube radius to be small enough so that every grid point within the computational tube \reminder{has} a unique closest-point projection. Mathematically, this condition implies that $R<\kappa_{\max}^{-1}$, where $\kappa_{\max}$ is the maximal principal curvature of $\Gamma$. On the other hand, the tube radius has to be large enough to include the whole discretization stencil. In this work, we fix $R=3\Delta x$.

\vspace{0.2cm}
\noindent
\textsf{Step 3}: Construct the push-forward matrix $\left[I - \phi(\vec{x})H(\vec{x})\right]^{-1}$. For every grid point $\vec{x} \in \Gamma^{R}$, we compute the Hessian matrix of $\phi$ using central difference, and therefore $\left[I - \phi(\vec{x})H(\vec{x})\right]$. We determine its inverse and store it on the \reminder{disk} for later computation.

\vspace{0.2cm}
\noindent
\textsf{Step 4}: Extend $u_0$ from $\Gamma$ to $\Gamma^R$. If the initial data $u_0$ is smooth enough, we use a high-order interpolation to compute $u_0(P_\Gamma(\vec{x}))$. However, for a piecewise smooth initial condition, we avoid interpolation and analytically determine the location of the projection of each grid point in the computational tube and directly evaluate the function value.

\vspace{0.2cm}
\noindent
\textsf{Step 5}: Solve the embedding conservation law \eqref{eq:EmbeddingConservationLaw}. Since equation \eqref{eq:EmbeddingConservationLaw} is a Cartesian conservation law posed in a subset of $\R^d$, we can apply existing numerical methods. In the numerical experiments, we consider the \reminder{Forward Euler method} with the Lax-Friedrichs scheme as the first-order method and the TVDRK3-WENO3 as the third-order method \cite{shu1998essentially}. We have chosen the Lax-Friedrichs flux splitting method in the WENO3. At every time step $t^n$, we first compute $\vec{F}(\vec{x}, u^n)$ for every grid point $\vec{x} \in \Gamma^R$, and then determine $\Tilde{\vec{F}}(\vec{x}, u^n)$. The numerical fluxes are computed from $\Tilde{\vec{F}}(\vec{x}, u^n)$. In particular, let $\Tilde{\vec{F}} = (f, g)$, we have the following discretizations. The \reminder{Forward Euler method} with the Lax-Friedrichs method is given by
\begin{equation*}
\frac{u^{n+1}_{i,j} - u^{n}_{i,j}}{\Delta t} +
\frac{f_{i+1/2, j} - f_{i-1/2, j}}{\Delta x} +
\frac{g_{i, j+1/2} - g_{i, j-1/2}}{\Delta y} = 0 \, ,
\end{equation*}
where
\begin{align*}
f_{i-1/2, j} &= \frac{1}{2} \left[ f_{i-1,j} + f_{i,j} - \frac{\Delta x}{d \Delta t} (u_{i,j} - u_{i-1,j}) \right] \, , \\
g_{i,j-1/2} &= \frac{1}{2} \left[ g_{i,j-1} + g_{i,j} - \frac{ \Delta y }{d \Delta t} (u_{i,j} - u_{i,j-1}) \right]
\end{align*}
with the dimension of the problem given by $d$, and $\Delta t$ satisfying the CFL condition
$$
\Delta t \max_u \left[ \frac{|f'(u)|}{\Delta x} + \frac{|g'(u)|}{\Delta y} \right] = \mbox{CFL} \le 1 \, ,
$$
while the third-order method is given by
\begin{align*}
u^{(1)}_{i,j} &= u^{n}_{i,j} + \Delta t \mathcal{L}(u_{i,j}^n), \nonumber \\
u^{(2)}_{i,j} &= \frac{3}{4}u^{n}_{i,j} + \frac{1}{4}\left(u^{(1)}_{i,j} + \Delta t \mathcal{L}(u^{(1)}_{i,j}) \right), \\
u^{n+1}_{i,j} &= \frac{1}{3}u^{n}_{i,j} + \frac{2}{3}\left(u^{(2)}_{i,j} + \Delta t \mathcal{L}(u^{(2)}_{i,j}) \right). \nonumber
\end{align*}
Here, $\mathcal{L}(u_{i,j}) = - \frac{1}{\Delta x}\left(\hat{f}_{i+1/2,j} - \hat{f}_{i-1/2,j}\right) - \frac{1}{\Delta y}\left(\hat{g}_{i,j+1/2} - \hat{g}_{i,j-1/2}\right)$, and $\hat{f}$ and $\hat{g}$ are the WENO3 fluxes.

When solving the advection equation on $\Gamma$, we follow the same idea and use the push-forward matrix to modify the propagation velocities:
\begin{align} \label{eq: EmbeddingAdvection}
\begin{split}
& \Tilde{u}_t + \left[I - \phi(\vec{x})H(\vec{x})\right]^{-1} \vec{V}(P_{\Gamma}(\vec{x})) \cdot \grad \Tilde{u} = 0, \\
& \Tilde{u}(\vec{x}, 0) = u_0(P_{\Gamma}(\vec{x})) \, .
\end{split}
\end{align}
Numerically, we implement the \reminder{Forward Euler method} with the upwind scheme as the first-order method and the TVDRK3 with the standard HJ-WENO3 scheme as the third-order method.

Regarding the boundary conditions on the computational tube, we impose the Neumann boundary condition $\partial_{\vec{n}} u = 0$ on $\partial \Gamma^R$, with $\vec{n} = \grad \phi$ denoting the outward unit normal vector. In the numerical implementation, we follow the embedding idea in \cite{wong2016fast} and introduce a thin computation layer, $(\Gamma^{R'} \setminus \Gamma^{R}) = { R \leq |\phi| < R'}$, and solve the extension equation $u_t + \mbox{sign}(\phi) \vec{n}\cdot \grad u = 0$ up to the steady-state in all intermediate time steps.

\section{Numerical Examples}
\label{Sec:Examples}

In this section, we consider various two- and three-dimensional examples to demonstrate the effectiveness and performance of the proposed algorithm. In all numerical examples, the computational domain is chosen to be a square/cube of side-length $2$ centered at the origin, and we choose the radius of the inner tube $R = 3\Delta x$ and the outer tube $R' = 8\Delta x$. Unless specified otherwise, the CFL number is always chosen to be $0.5$ for both first- and third-order methods. The matrix $\left[ I - \phi(\vec{x})H(\vec{x}) \right]$ is computed by second-order central difference, even though we can explicitly construct the analytical expression when the signed distance function is known. In our examples, we found that the error introduced does not affect the convergence rate.

To study the convergence behavior of our proposed approach, we measure the $L^1$-, $L^2$-, and $L^\infty$-errors on $\Gamma$. Since we compute the solutions on the Cartesian mesh, we first interpolate these gridded solutions on a fine surface mesh by the high-order \textsf{cubic} interpolation implemented using \textsf{MATLAB} \textsf{interp2} (for curves in two dimensions) or \textsf{interp3} (for surfaces in three dimensions). Then we integrate these values appropriately using the Trapezoidal rule to obtain approximations of these errors.


\subsection{Advection Equation on the Unit Circle}
\label{SubSec:ExAdvectionCircle}

\begin{figure}[!ht]
    \centering
    \includegraphics[width=1\textwidth]{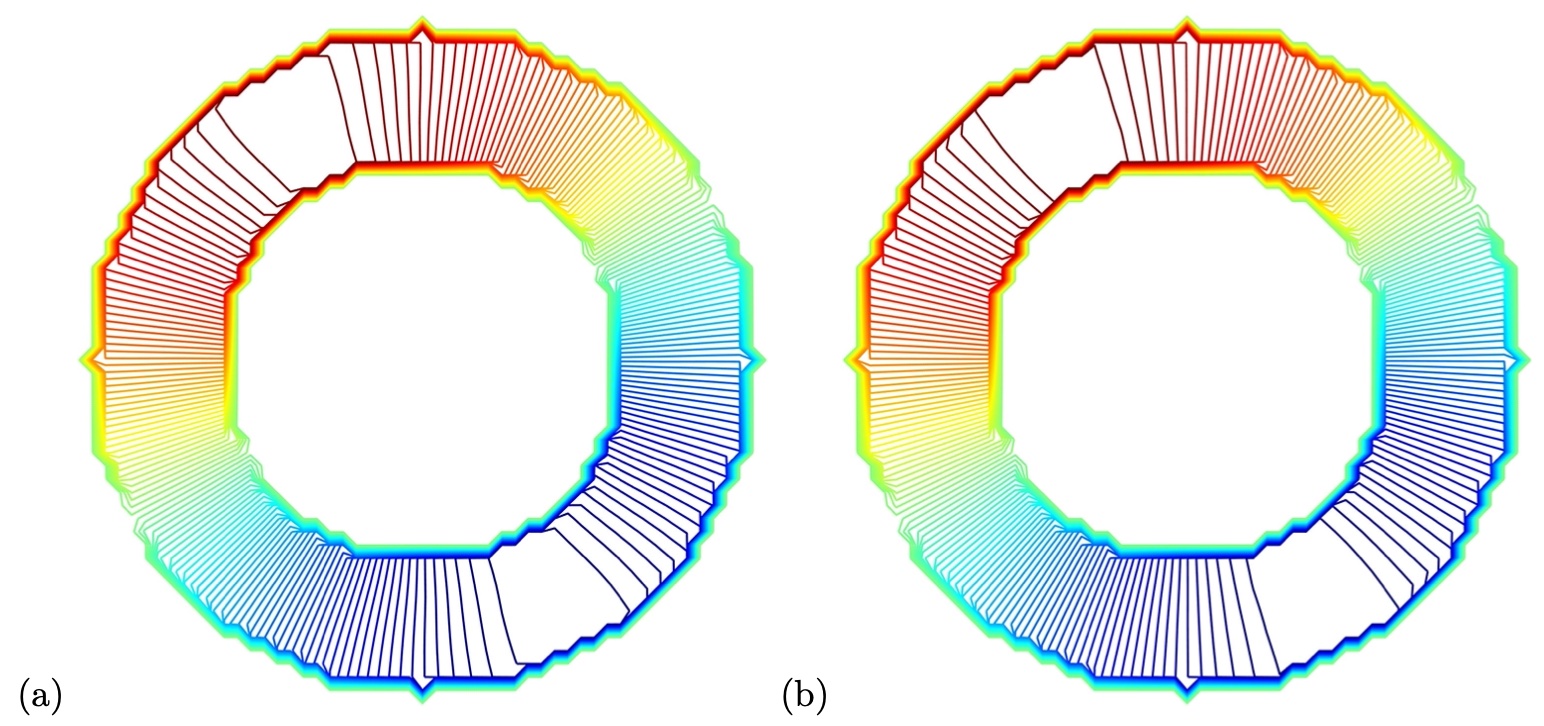}
    \caption{(Section \ref{SubSec:ExAdvectionCircle} Advection equation on the unit circle) Contour plot\reminder{s} of the numerical solutions computed using \reminder{\reminder{(a) the first- and (b) the third-order methods}s}. We choose $t = 0.5$ and $\Delta x = 0.05$.}
    \label{fig: Contour Advection on circle}
\end{figure}


\begin{figure}[!ht]
    \centering
    \includegraphics[width=1\textwidth]{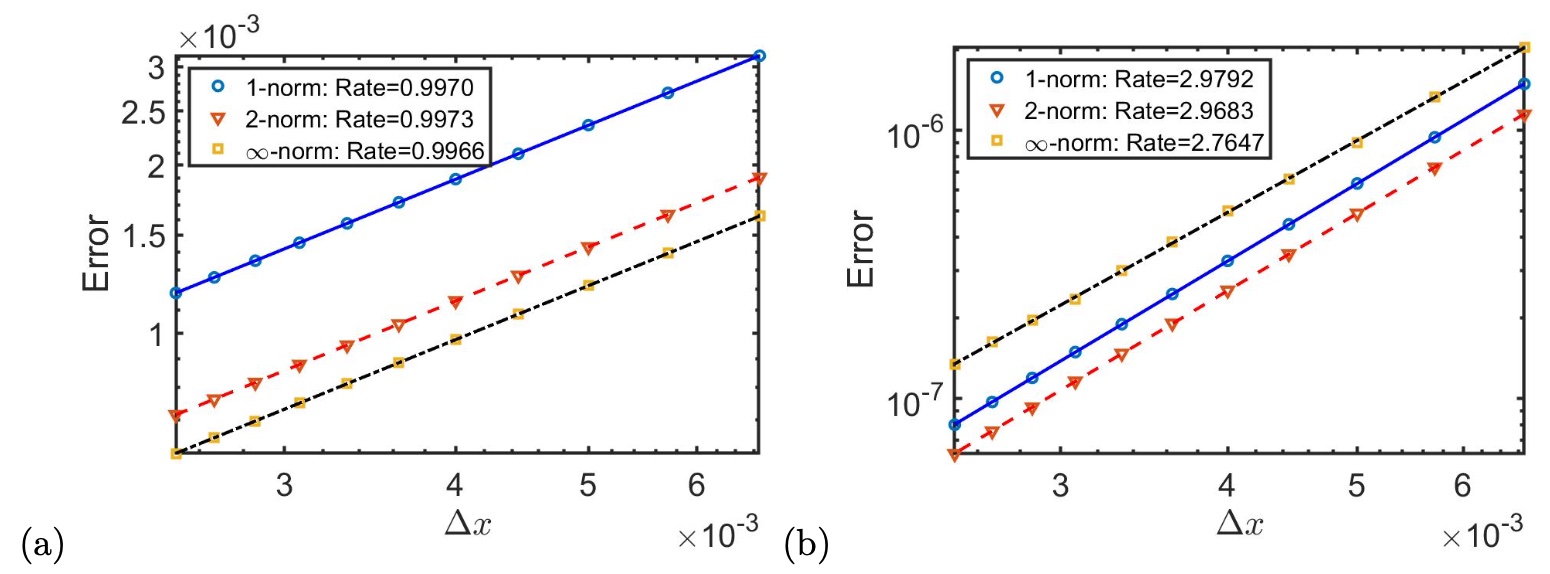}
   \caption{ (Section \ref{SubSec:ExAdvectionCircle} Advection equation on the unit circle) The numerical errors in the solutions computed using \reminder{\reminder{(a) the first- and (b) the third-order methods}s}. The final time is chosen to be $t = 0.5$.}
    \label{fig: CircleAdvection}
\end{figure}

First, we consider the advection equation on the unit circle centered at the origin, $u_t + u_\theta = 0$, where $\theta$ is the polar angle. The initial condition is given by $u_0(\theta) = \sin{\theta}$. This problem can be regarded as the advection on the $[0, 2\pi]$ interval with periodic boundary conditions. The velocity field is given by $\vec{V}(\vec{x}) = (-\sin\theta, \cos\theta)$, and the modified velocity field $\Tilde{\vec{V}}$ can be computed by left-multiplying the push-forward matrix.

Figure \ref{fig: Contour Advection on circle} shows the contour plots of the numerical solutions using our proposed approach. We see that the solutions are constant along normal directions as desired. In Figure \ref{fig: CircleAdvection}, we show the numerical errors in our solutions computed on various meshes. We can see that the convergence rates of both errors are approximately $1$ and $3$ when using the first- and the third-order methods, respectively.


\subsection{Advection Equation on an Ellipse}
\label{SubSec:ExAdvectionEllipse}

\begin{figure}[!ht]
    \centering
    \includegraphics[width=\textwidth]{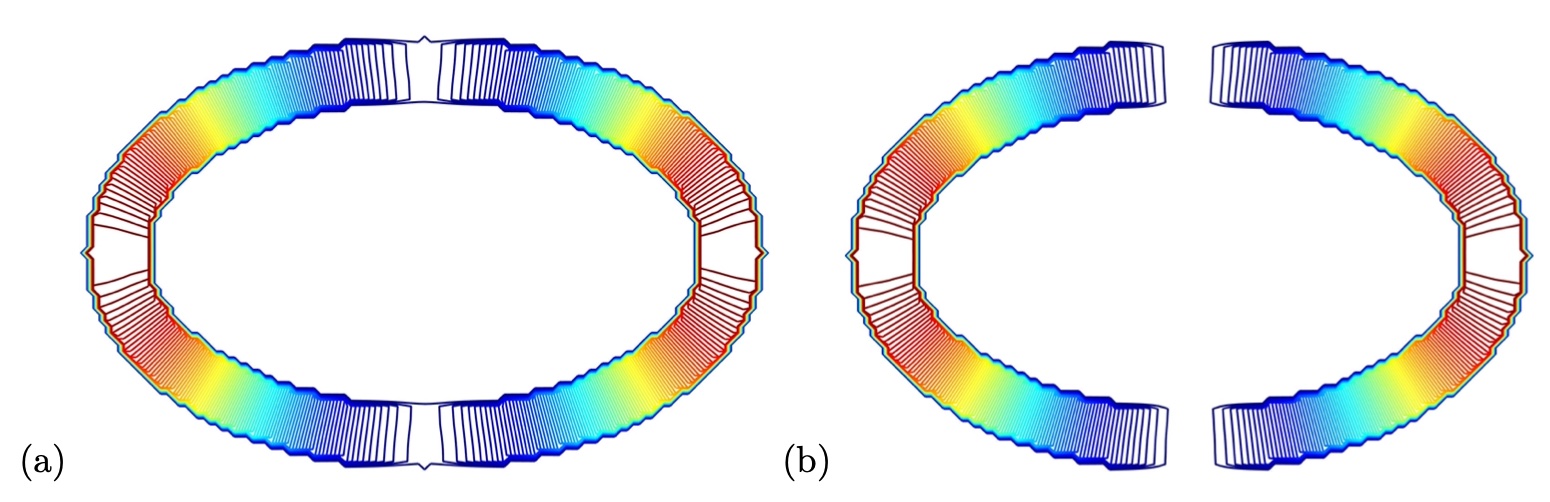}
    \caption{(Section \ref{SubSec:ExAdvectionEllipse} Advection equation on an ellipse) Contour plot\reminder{s} of the numerical solutions computed using \reminder{(a) the first- and (b) the third-order methods}. We choose $t = L/4$ and $\Delta x = 0.025$.}
    \label{fig: Contour Advection on ellipse}
\end{figure}

\begin{figure}[!ht]
    \centering
     \includegraphics[width=\textwidth]{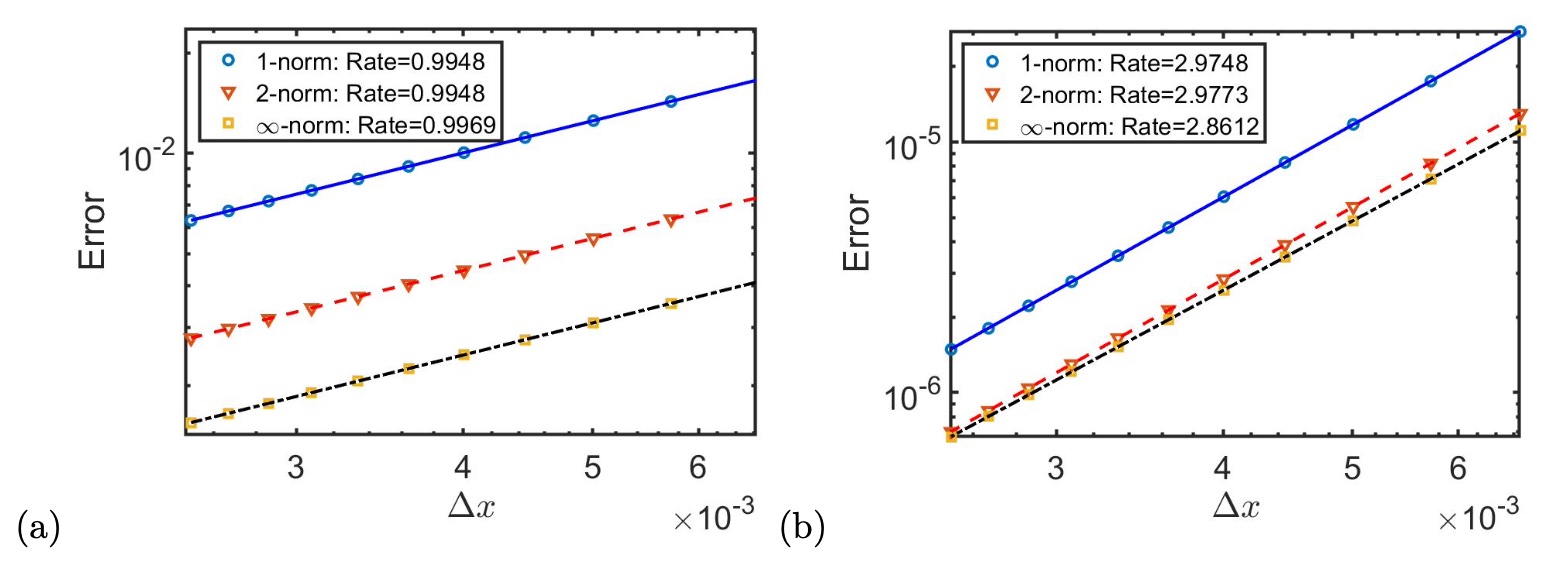}
   \caption{ (Section \ref{SubSec:ExAdvectionEllipse} Advection equation on \reminder{an} ellipse) The numerical errors in the solutions computed using \reminder{(a) the first- and (b) the third-order methods}. The final time is chosen to be $t = L/4$.}
    \label{fig: EllipseAdvection}
\end{figure}

Next, we consider \reminder{an} ellipse given by $\sqrt{x^2/a^2 + y^2/b^2} = 1$ with $a = 0.75$ and $b=1.25$ and solve the advection equation $u_t + u_{s} = 0$ on the interface where $s$ is the arclength of the ellipse. Since the above representation does not give a signed distance function, we first let $f(\theta;x,y) = \frac{1}{2} \| (x,y) - (a\cos \theta, b \sin \theta)\|^2$ and determine $\theta^*$ such that $\theta^*=\mbox{argmin}_{\theta} f(\theta;x,y)$. Then, the signed distance function representation is given by
$$
\phi(x,y)=\sign{\sqrt{x^2/a^2 + y^2/b^2} - 1} \| (x,y) - (a\cos \theta^*, b \sin \theta^*)\| \, .
$$

We pick the initial condition $u_0(s) = \cos^2 (2\pi s/L)$ as in \cite{ruuth2008simple}, where $L$ is the perimeter of the ellipse. The embedding problem is given by
$$
u_t + \left[I - \phi(\vec{x})H(\vec{x})\right]^{-1}\vec{V}(\vec{x})\cdot \grad u = 0, \mbox{ where }
\vec{V}(\vec{x}) = \left(-\frac{y}{b^2}, \frac{x}{a^2}\right)/\sqrt{\frac{x^2}{a^4}+\frac{y^2}{b^4}}
$$
indicates the advection directions. Figure \ref{fig: Contour Advection on ellipse} shows contour plots for the advection equation on the ellipse, and the numerical solutions are constant along the normal directions of the ellipse. Figure \ref{fig: EllipseAdvection} shows the $L^1$-, $L^2$-, and $L^\infty$-convergence rates of the numerical solutions computed using the first- and the third-order methods, which are approximately $1$ and $3$, respectively.


\subsection{Burgers' Equation on the Unit Circle}
\label{SubSec:ExBurgerCircle}

\begin{figure}[!ht]
\centering
\includegraphics[width=\textwidth]{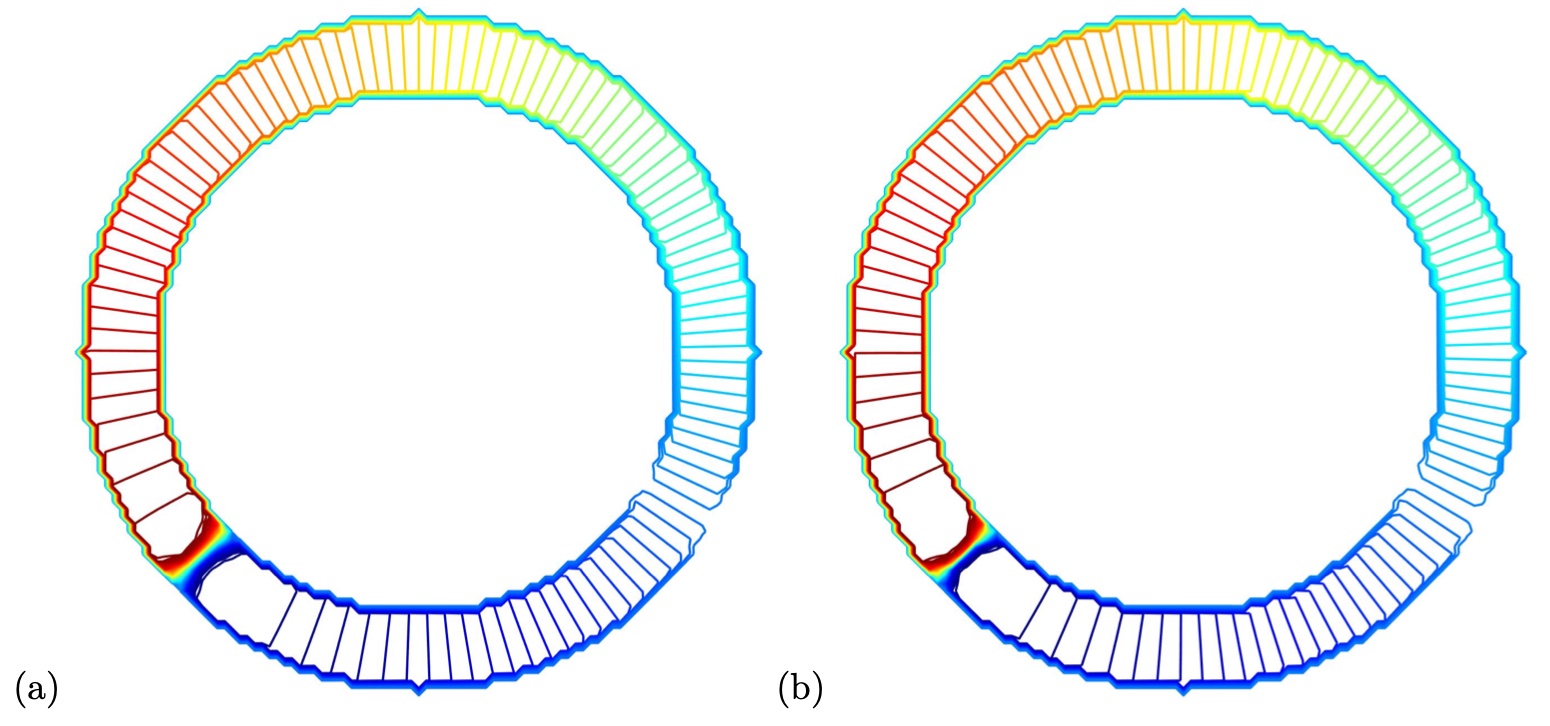} 
\caption{(Section \ref{SubSec:ExBurgerCircle} Burgers' equation on the unit circle) Contour plot\reminder{s} of the numerical solutions computed using \reminder{(a) the first- and (b) the third-order methods}. We choose $t = 1.5$ and $\Delta x = 0.025$.}

\label{Fig:BurgerCircle-Shock}
\end{figure}    

\begin{figure}[!ht]
    \centering
    \includegraphics[width=\textwidth]{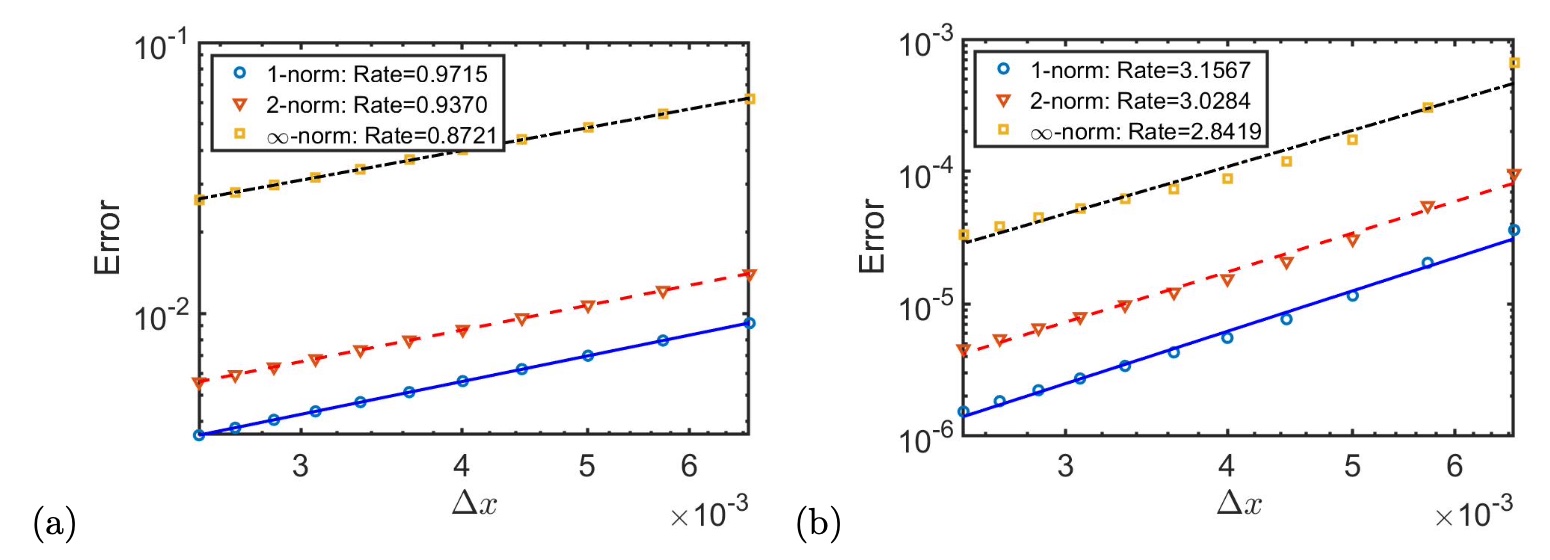}
    \caption{ (Section \ref{SubSec:ExBurgerCircle} Burgers' equation on the unit circle) The numerical errors in the solutions computed using \reminder{(a) the first- and (b) the third-order methods}. The final time is chosen to be $t = 0.9$.} 
    \label{fig: CircleBurgers}
\end{figure}   

We consider \reminder{Burgers' equation} on the unit circle, given by $u_t + \frac{1}{2}\left(u^2\right)_{\theta} = 0$. The problem can be reformulated as \reminder{Burgers' equation} on $[0,2\pi]$ with periodic boundary conditions, and therefore, we can write the solution implicitly as $u(\theta,t) = u_0(\theta - ut)$.

In the numerical example, we choose the initial condition $u_0(\theta) = \sin \theta + 0.5$. Even with this smooth initial condition, the solution develops a singularity, and the shock forms at $t_s = 1$. Figure \ref{Fig:BurgerCircle-Shock} shows the numerical solutions obtained by the first- and the third-order methods at the time $t=1.5$ computed on a mesh of $\Delta x = 0.025$. We see that these solutions are constant along the normal directions. The solution develops a shock at time $t>1$. The proposed method can capture the discontinuity, and the resolution improves as the order of the numerical method increases. To demonstrate the accuracy of our numerical solution, we compute the numerical solution at $t=0.9$ when the profile is still smooth. Figure \ref{fig: CircleBurgers} shows the numerical errors measured in various norms. We see that the first- and the third-order methods converge at rates of $1$ and $3$, respectively.



\begin{figure}[!ht]
    \centering
    \includegraphics[width=\textwidth]{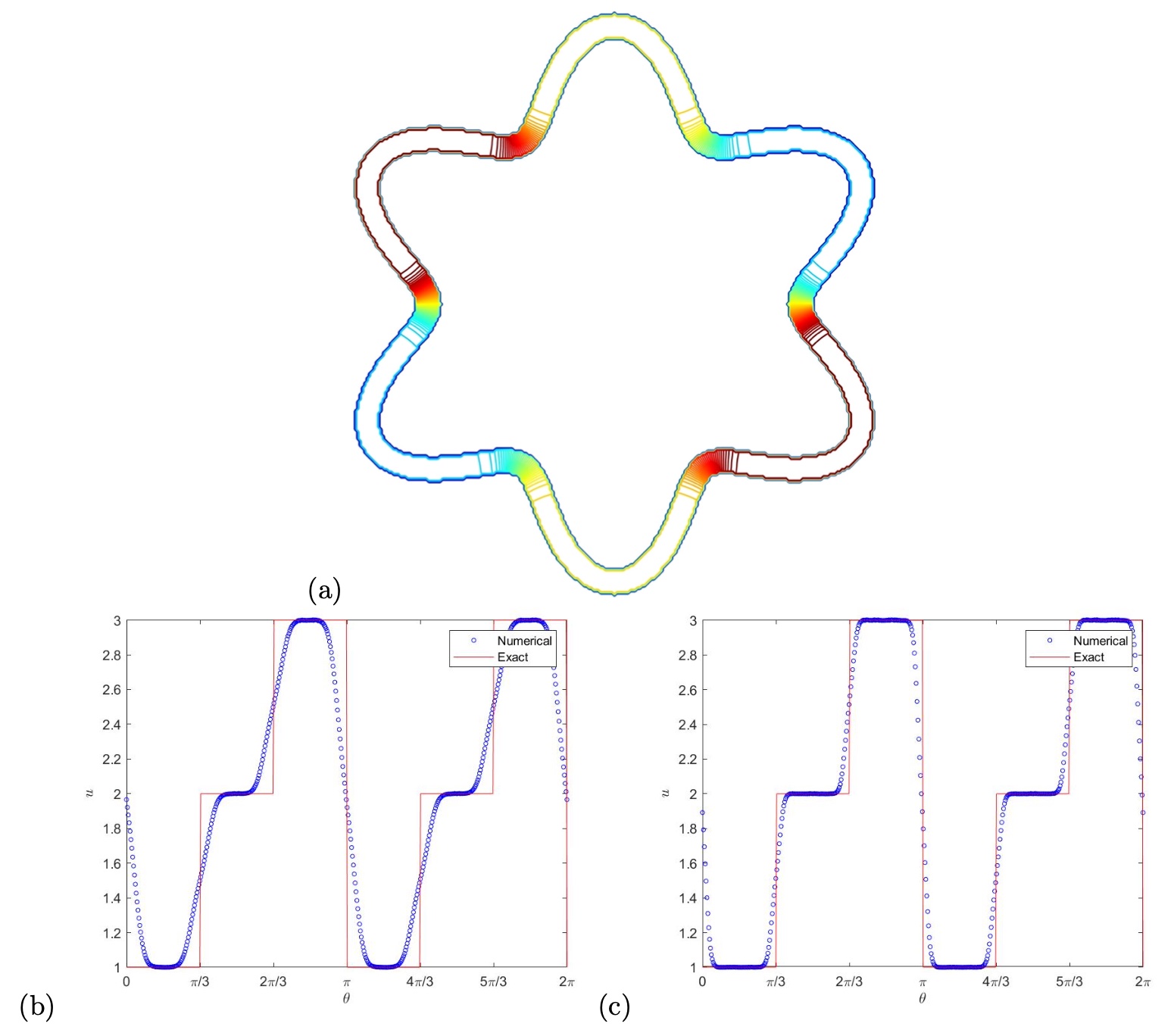}
    \caption{(Section \ref{SubSec:ExAdvectionStar} Advection on a star curve) (a) Contour plot of the numerical solution obtained by the third-order method. We also interpolate the numerical solutions on the curve computed using (b) the first- and (c) the third-order methods. We choose $t = L$ and $\Delta x = 0.0125$.}
    \label{Fig:AdvectionStar}
\end{figure}

\subsection{Advection on a Star Curve}
\label{SubSec:ExAdvectionStar}

In this example, \reminder{we consider a star curve given by the polar equation:} $r(\theta) = r_0 + a\sin^2{(b\theta)}$, where $r_0 = 1, a = 0.5$ and $b = 3$. The corresponding level set function is given by $\phi(x,y) = \sqrt{x^2 + y^2} - r_0 - a\sin^2{(b\arctan{(y/x)} )}$. \reminder{Since} $\phi$ is not a signed distance function, we first compute the closest-points of grids and define $\Tilde{\phi}(x,y) = \sign{\phi(x,y)}\|(x,y) - P_{\Gamma}(x,y)\|$. The tangent vector of the star curve can be computed by differentiating $(r(\theta)\cos{\theta}, r(\theta)\sin{\theta})$, where $\theta = \arctan(y/x)$. Hence, we define the velocity field $\vec{V}(x,y)$ as
\[
\vec{V}(x,y) :=
\frac{
\left(-y + ab\sin{(2b\theta)}\cos{\theta}, x + ab\sin{(2b\theta)}\sin{\theta} \right)
}{
\|\left(-y + ab\sin{(2b\theta)}\cos{\theta}, x + ab\sin{(2b\theta)}\sin{\theta} \right)\|
}
\]
for $(x,y) \in \Gamma$. The embedding advection equation is $u_t + \left[I - \phi(\vec{x})H(\vec{x})\right]^{-1}\vec{V}(x,y)\cdot \grad u = 0$. In contrast to previous examples, we consider an initial condition with discontinuities, given by
\begin{equation*}
u_0(\theta) =
\begin{cases}
1, & \theta \in [0, \pi/3) \cup [\pi, 4\pi/3)\\
2, & \theta \in [\pi/3, 2\pi/3) \cup [4\pi/3, 5\pi/3)\\
3, &\theta \in [2\pi/3, \pi) \cup [5\pi/3, 2\pi) \, .
\end{cases}
\end{equation*}
At $t = L$, where $L \approx 10.1964221$ is the perimeter of the star curve, the above advection is a counter-clockwise rotation about the origin by $2\pi$ radians. From Figure \ref{Fig:AdvectionStar}, we can observe that the numerical solutions are constant along the normal directions and match the exact solution nicely.


\subsection{Advection on a Torus}
\label{SubSec:ExAdvectionTorus}

\begin{figure}[!ht]
    \centering
    \includegraphics[width=\textwidth]{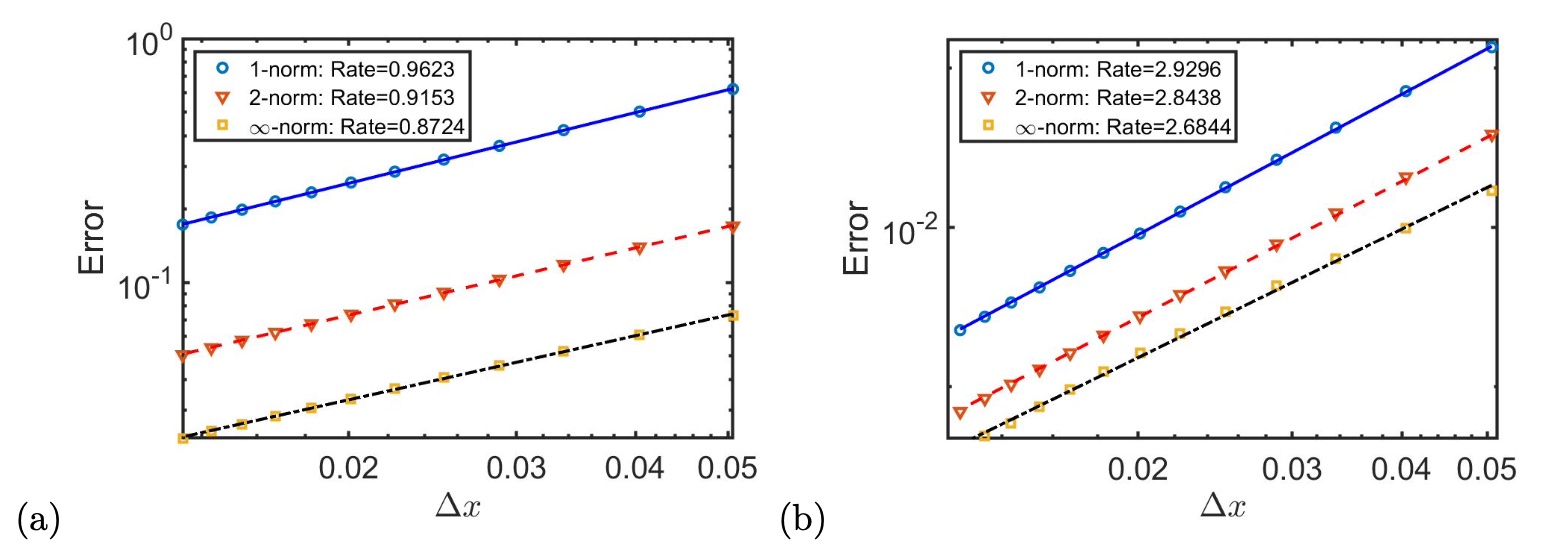}
    \caption{ (Section \ref{SubSec:ExAdvectionTorus} Advection equation on \reminder{a} torus) The numerical errors in the solutions computed using \reminder{(a) the first- and (b) the third-order methods}. The final time is chosen to be $t = 1$.}
    \label{fig: TorusAdvection}
\end{figure}

\begin{figure}[!ht]
\centering
\includegraphics[width=\textwidth]{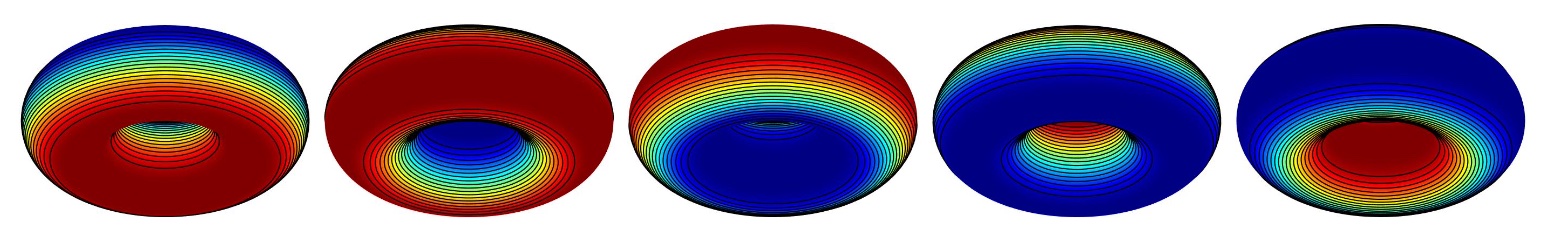}
\caption{(Section \ref{SubSec:ExAdvectionTorus} Advection on a torus) 
Numerical solutions at $t = 2\pi/5$, $4\pi/5$, $6\pi/5$, $8\pi/5$ and $2\pi$ computed using the third-order method.} 
\label{Fig:AdvectionTorus_Solution}
\end{figure}

Consider a torus in $\R^3$ given by the parametrization
$$
(x, y, z) = \left( (R + r\cos\eta)\cos\theta, (R + r\cos\eta)\sin\theta, r\sin\eta \right)
$$
where $\theta, \eta \in [-\pi, \pi]$, and $R = 1, r = 0.5$. The corresponding signed distance function is
$
\phi(x,y,z) = \left[z^2 + \left(\sqrt{x^2 + y^2 } - R\right)^2 \right]^{\frac{1}{2}} - r \, .
$
Compared to the unit sphere, \reminder{this} torus has two different principal curvatures. It is, therefore, a more challenging example to test the performance of our algorithm. Similar to \cite{greer2006improvement,ruuth2008simple}, we solve the advection equation $u_t + u_{\eta} = 0$ with the initial condition
\begin{equation*}
u(\theta, \eta, 0) = f(\eta) =
\begin{cases}
g\left(\frac{\pi + \eta}{\pi}\right), \quad \eta \in [-\pi, 0],\\
g\left(\frac{\pi - \eta}{\pi}\right), \quad \eta \in (0, \pi],
\end{cases}
\end{equation*}
where $g(x) := \left( e^{\frac{1}{x-1}} - e^{\frac{1}{x}} \right) / \left( e^{\frac{1}{x-1}} + e^{\frac{1}{x}} \right)$.
The exact solution is given by $u(\theta, \eta, t) = f(\eta - t)$. In Figure \ref{fig: TorusAdvection}, we can observe the first- and the third-order convergence in both errors when using the first- and the third-order methods, respectively. We have plotted the solution at different times in Figure \ref{Fig:AdvectionTorus_Solution}. The solution well captures the evolution of the smooth profile.

\subsection{Burgers' Equation on the Unit Sphere}
\label{SubSec:ExBurgerSphere}

\begin{figure}[!ht]
    \centering
    \includegraphics[width=\textwidth]{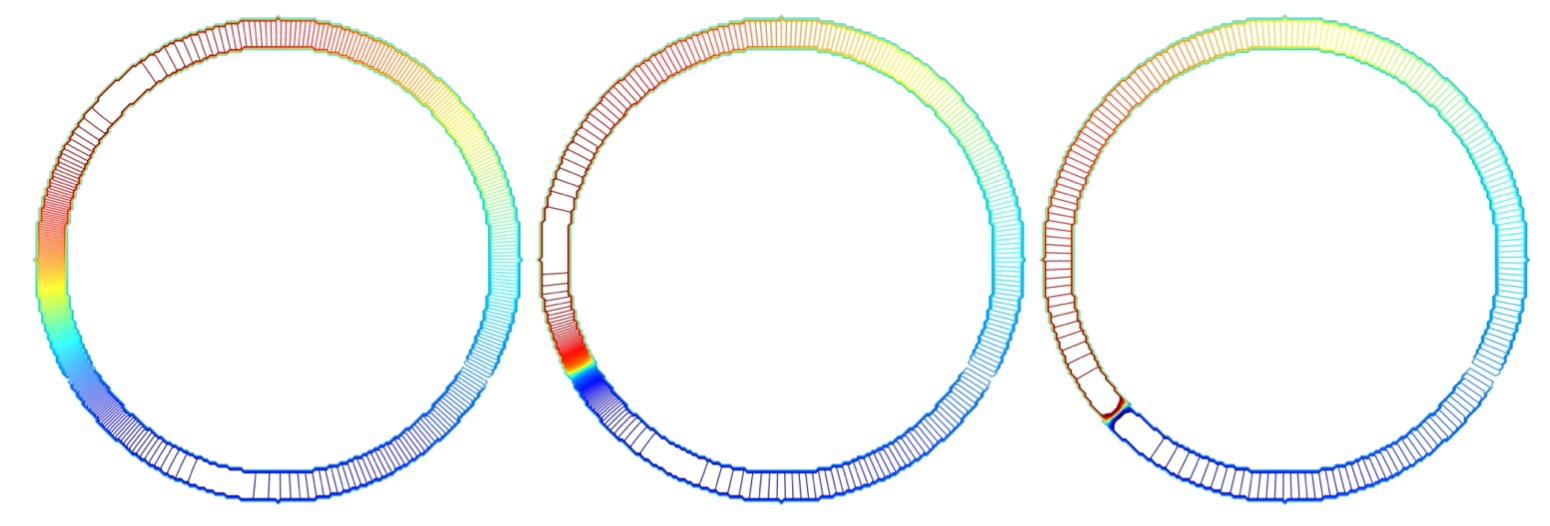} 
    \caption{(Section \ref{SubSec:ExBurgerSphere} Burgers' equation on the unit sphere with the initial condition $u_1$) Contour plot\reminder{s} of the numerical solutions on $z = 0$ obtained by the third-order method at $t=0.5, 1$ and $1.5$.}
    \label{Fig:BurgerSphere}
\end{figure}
    
\begin{figure}[!ht]
    \centering
    \includegraphics[width=\textwidth]{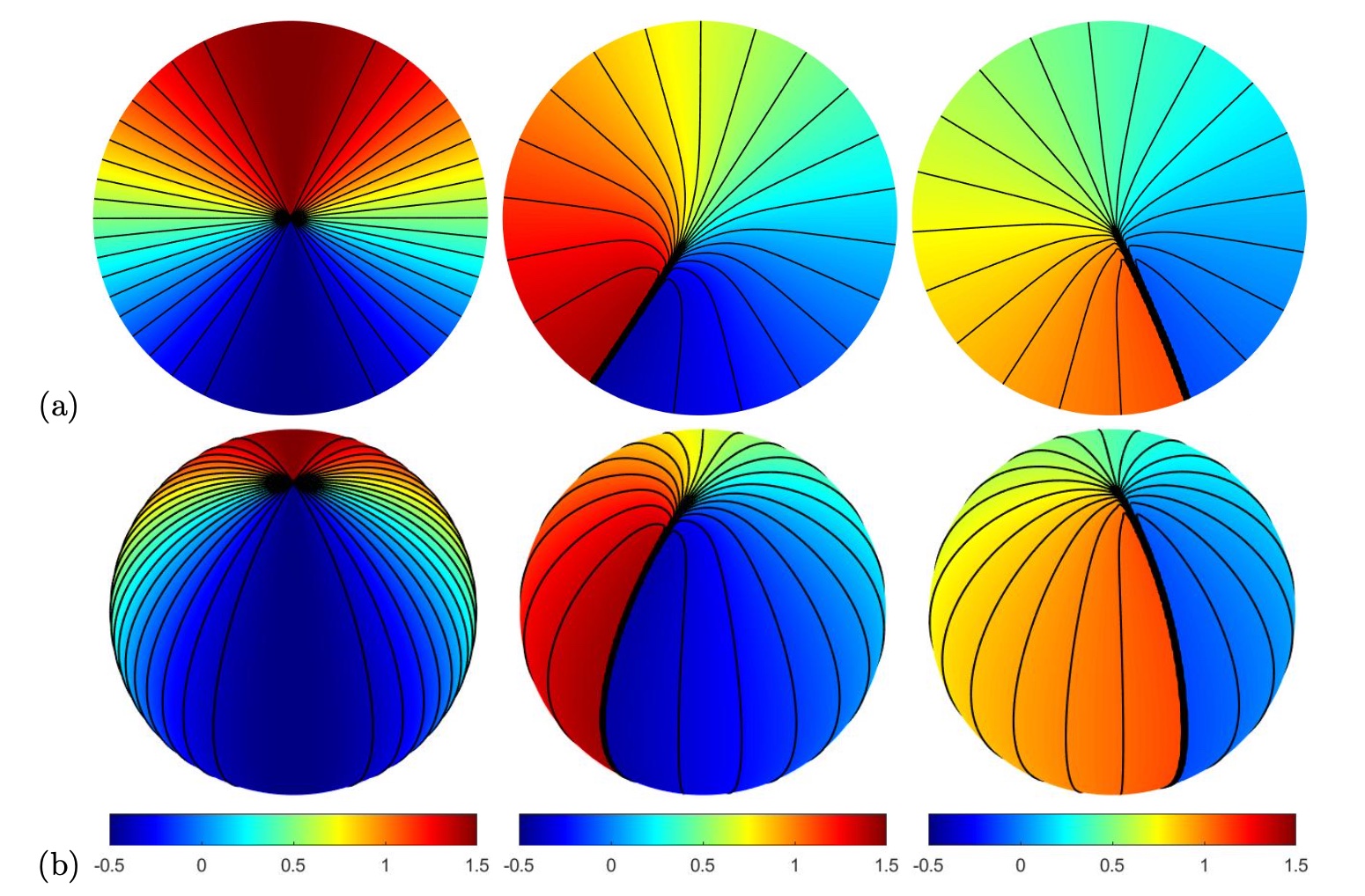} 
    \caption{(Section \ref{SubSec:ExBurgerSphere} Burgers' equation on the unit sphere with the initial condition $u_1$, WENO3) Numerical solutions at $t = 0, 2$ and $4$ with $u_1$ as the initial condition and $\Delta x=0.0125$. (a-b) Solutions observed from different angles. }
    \label{Fig:BurgerSphereShock_Solution_0.0125_WENO}
\end{figure}

\begin{figure}[!ht]
\centering
\includegraphics[ width=\textwidth]{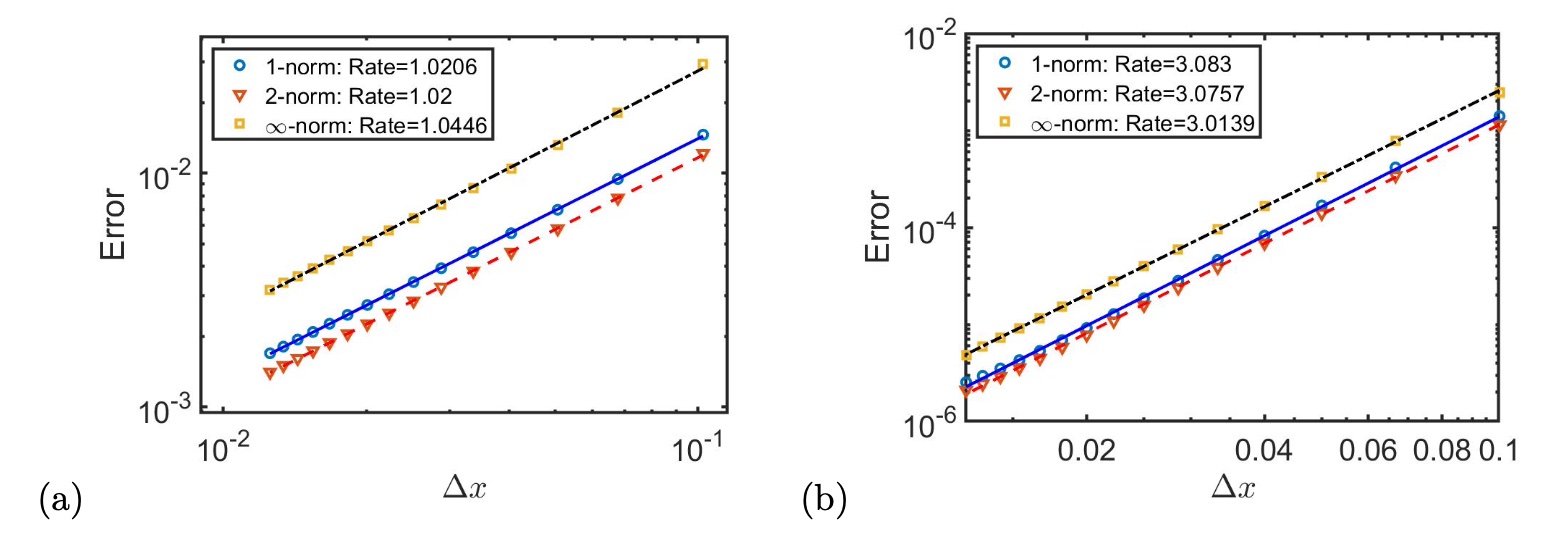}
\caption{ (Section \ref{SubSec:ExBurgerSphere} Burgers' equation on the unit sphere with the initial condition $u_1$) The numerical errors in the solutions computed using \reminder{(a) the first- and (b) the third-order methods}. The final time is chosen to be $t = 0.5$}
\label{fig:BurgerSphereConvergence}
\end{figure}

\begin{figure}[!ht]
    \centering
    \includegraphics[width=\textwidth]{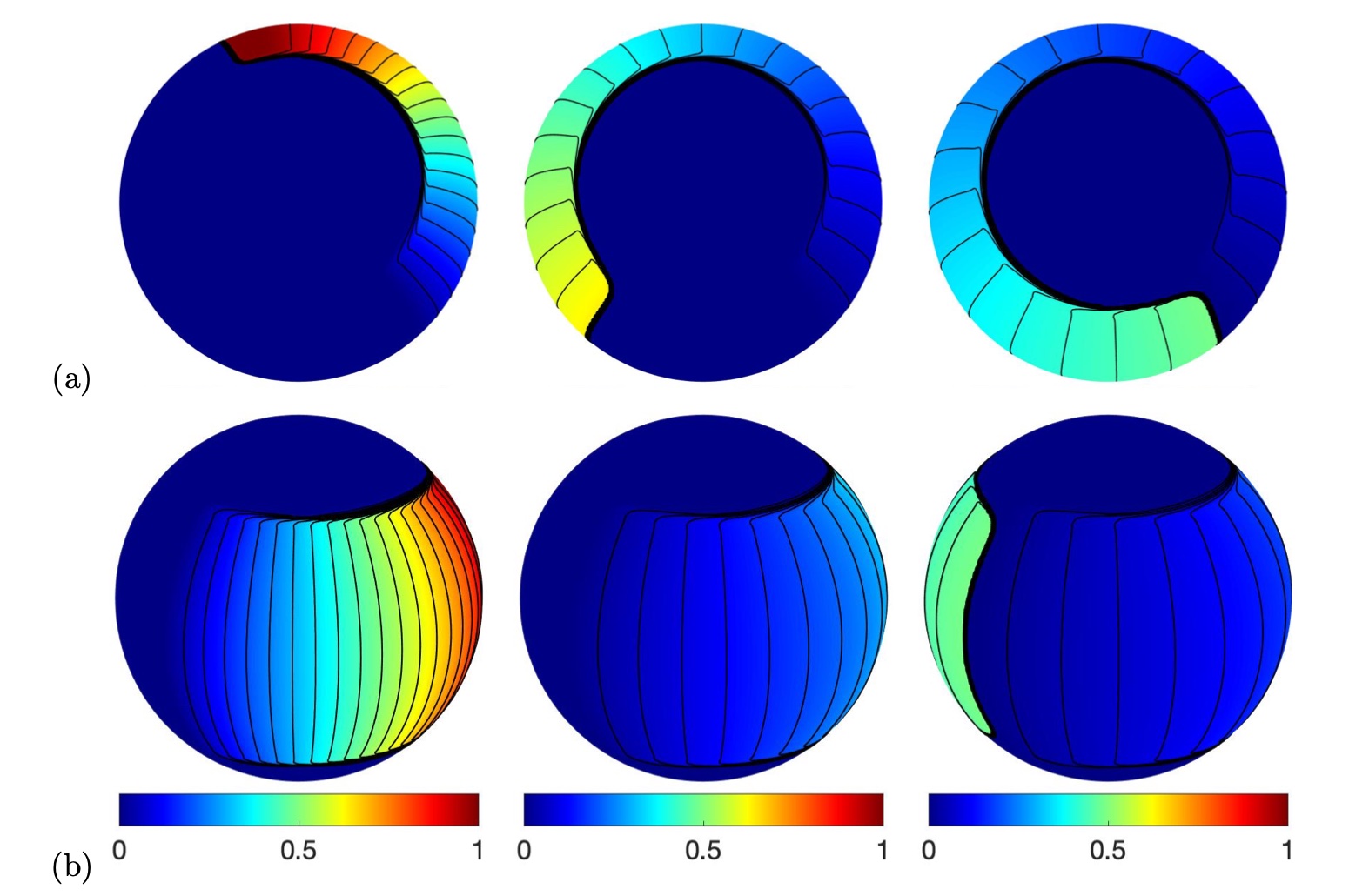} 
    \caption{(Section \ref{SubSec:ExBurgerSphere} Burgers' equation on the unit sphere with the initial condition $u_2$, WENO3) Numerical solutions at $t = 4\pi/5, 12\pi/5$ and $4\pi$ with $u_2$ as the initial condition and $\Delta x=0.0125$. (a-b) Solutions observed from different angles.}
    
    \label{Fig:BurgerSphereBox_Solution_0.0125_WENO}
    \end{figure}

\begin{figure}[h!]
\centering
\includegraphics[width=\textwidth]{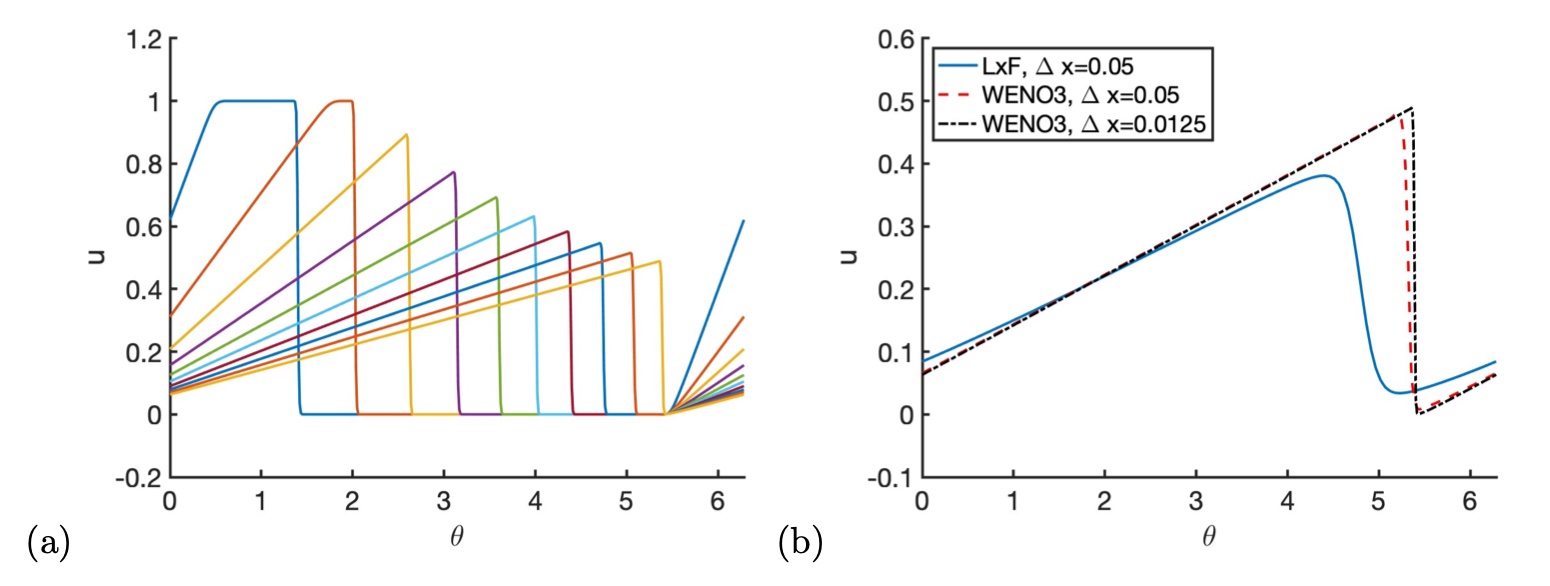}
\caption{(Section \ref{SubSec:ExBurgerSphere} Burgers' equation on the unit sphere with the initial condition $u_2$) (a) Cross\reminder{-}sections on $z=0$ at $t = \pi/5$ to $2\pi$ with an increment of $\pi/5$. The solution is computed on the mesh $\Delta x=0.0125$ using the third-order WENO method. (b) Comparison of the solutions at $t=2\pi$ on $z=0$ obtained by the first-order LxF method with $\Delta x =0.05$, the third-order WENO method with $\Delta x=0.05$ and $\Delta x=0.0125$.}

\label{Fig:BurgerSphereBox_Solution_CrossSection}
\end{figure}

\begin{figure}[!ht]
\centering
\includegraphics[width=\textwidth]{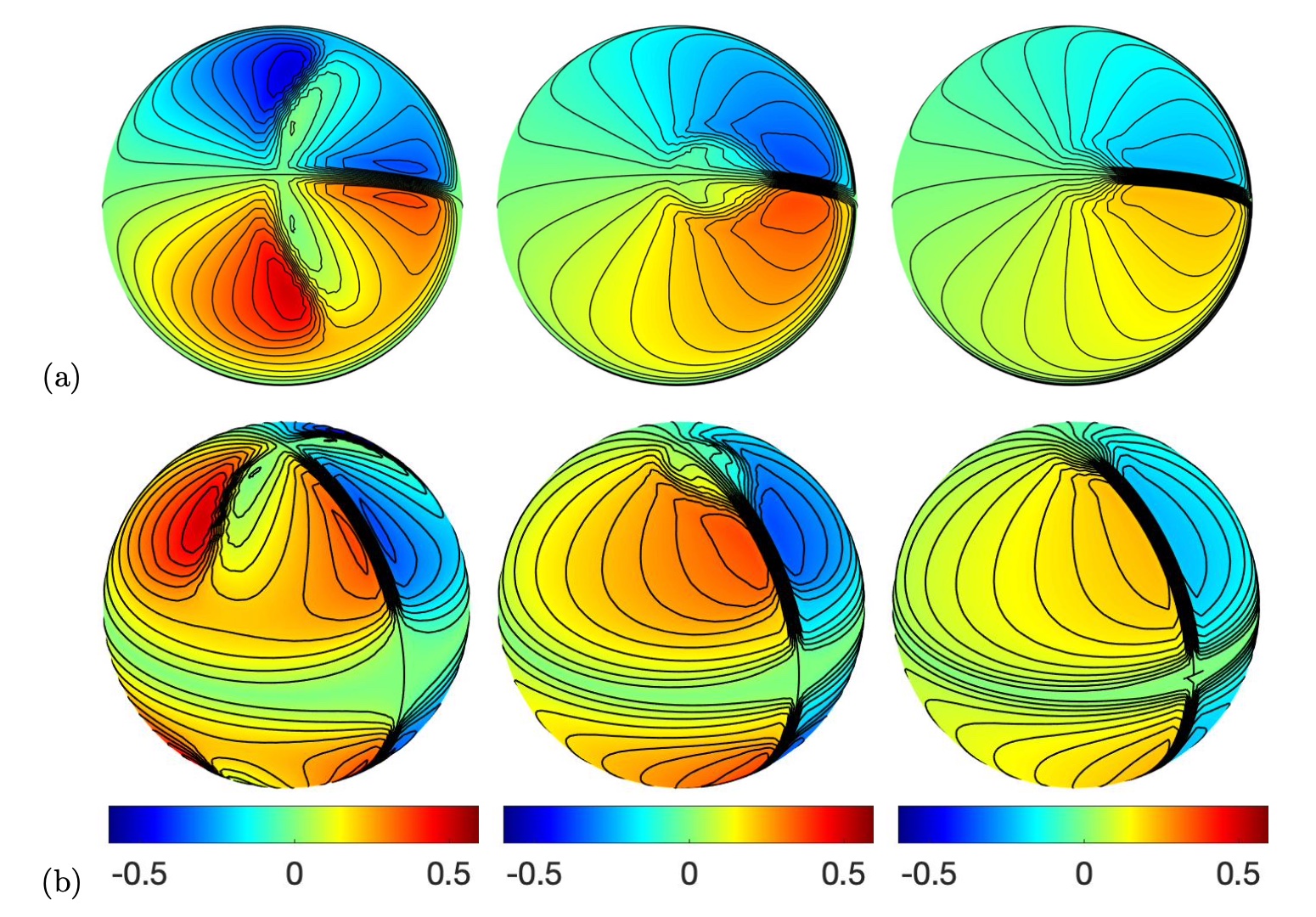} 
\caption{(Section \ref{SubSec:ExBurgerSphere} Burgers' equation on the unit sphere with the initial condition $u_3$, WENO3) Numerical solutions at $t = 4\pi/5, 12\pi/5$ and $4\pi$ with $\Delta x=0.05$. (a-b) Solutions observed from different angles.}
\label{Fig:BurgerSphereSH_Solution_05}
\end{figure}

\begin{figure}[!ht]
\centering
\includegraphics[width=\textwidth]{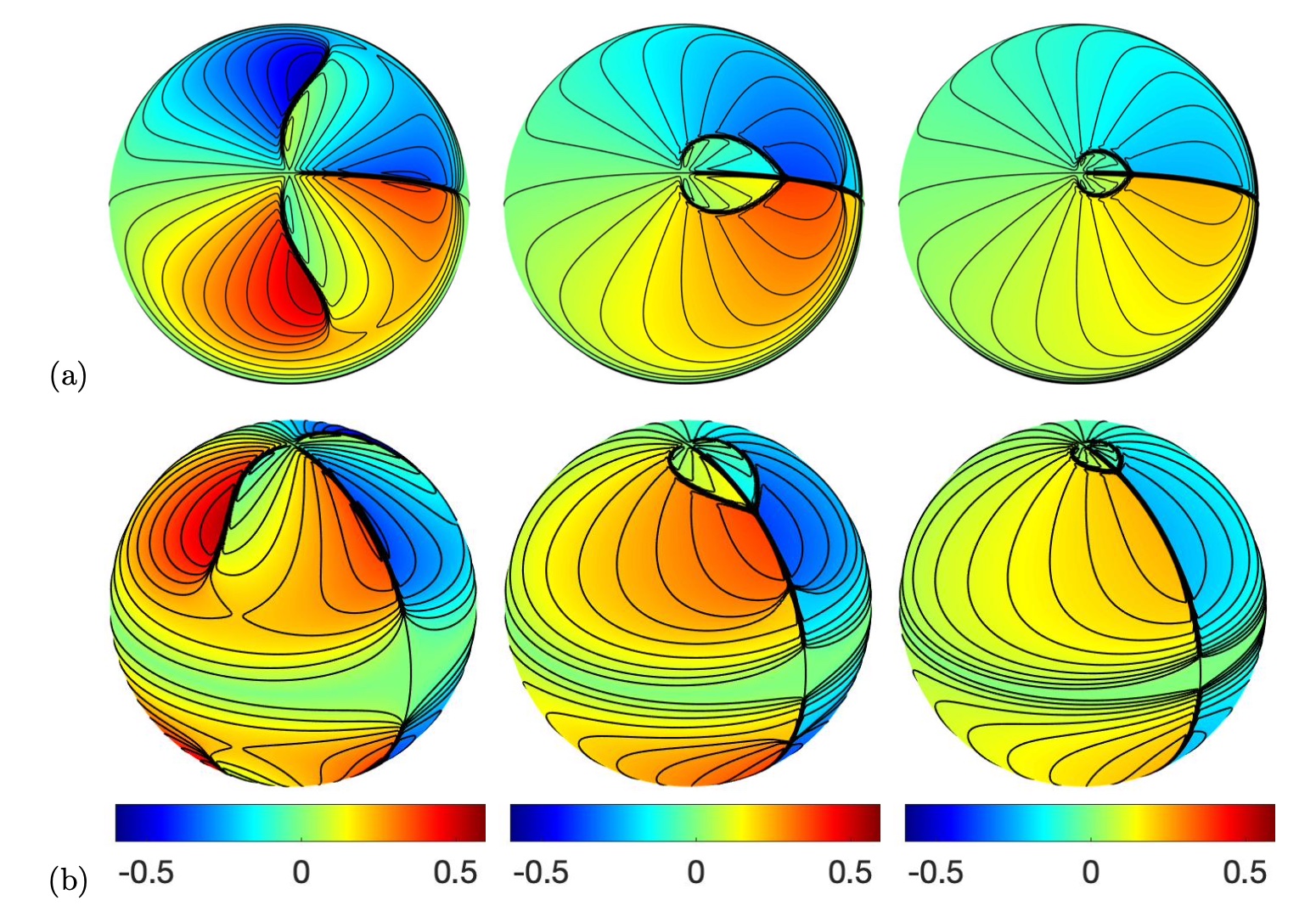} 
\caption{(Section \ref{SubSec:ExBurgerSphere} Burgers' equation on the unit sphere with the initial condition $u_3$, WENO3) Numerical solutions at $t = 4\pi/5, 12\pi/5$ and $4\pi$ with $\Delta x=0.0125$. (a-b) Solutions observed from different angles.
}
\label{Fig:BurgerSphereSH_Solution_0125}
\end{figure}

We consider \reminder{Burgers' equation} posed on the unit sphere, given by $u_t + \left(\frac{1}{2}u^2\right){\theta} = 0$, where $\theta$ is the azimuthal angle, with three different initial conditions:
\begin{itemize}
\item $u_1(\theta, \varphi, 0) = \sin\theta+0.5$,
\item $u_2$ is the characteristic function of a \textit{box} defined by $|\varphi|<\pi/4$ and $|\theta|<\pi/4$ where $\varphi$ and $\theta$ are the polar and the azimuthal angle in the spherical coordinate, respectively, and
\item $u_3$ is a combination of two spherical harmonic modes given by $Y_{2, -1}+Y_{4,-3}$.
\end{itemize}
The first initial condition provides a relatively simple test example so that we can construct an exact solution. Even though the solution develops a shock later, we can use the short-time solution to test the convergence of the numerical schemes. The second example is slightly more complicated, showing the interaction of a shock and a rarefaction wave. The third initial condition $u_3$ is the most challenging one. We observe the interaction of multiple shocks.

Figure \ref{Fig:BurgerSphere} shows the cross-\reminder{section} of the numerical solution at various times with the initial condition given by $u_1$ on the $z=0$ plane. The solution on the sphere is shown in Figure \ref{Fig:BurgerSphereShock_Solution_0.0125_WENO}. We see that the third-order method can well resolve the discontinuity, and the shock is sharp. We also observe that the solution is indeed constant along the normal directions of the surface. We also compute the error\reminder{s} in the solutions at $t=0.5$ obtained by various meshes to see the convergence rate\reminder{s}. From Figure \ref{fig:BurgerSphereConvergence}, we observe that the convergence rate\reminder{s} \reminder{are} roughly $1$ and $3$ for the first- and the third-order methods, respectively.

Figure \ref{Fig:BurgerSphereBox_Solution_0.0125_WENO} shows the numerical solutions with the initial condition given by $u_2$ and the mesh $\Delta x=0.0125$. Since the initial profile has a sharp discontinuity, the right boundary of the box moves as a shock towards the right-hand side. At the same time, the left border forms a rarefaction wave immediately. Because there is no flux across the upper and the boundary boundaries, these discontinuities should be stationary and do not move in the zenith direction (angle from the polar axis).

To compare these solutions more clearly, we extract the $z=0$ cross-section of the solution at various times in Figure \ref{Fig:BurgerSphereBox_Solution_CrossSection}(a). The rarefaction fan catches the solution slightly after $t=2\pi/5$, and the magnitude of the shock gradually reduces over time. In Figure \ref{Fig:BurgerSphereBox_Solution_CrossSection}(b), we plot all solutions at the final time $t=2\pi$. The accuracy of the third-order method is significantly better than that of the first-order LxF solution. The numerical scheme is highly dissipative.

Figures \ref{Fig:BurgerSphereSH_Solution_05}-\ref{Fig:BurgerSphereSH_Solution_0125} consider a more complicated case given by the initial condition $u_3$ with the mesh sizes $\Delta x=0.05$ and $\Delta x=0.0125$, respectively. Due to the symmetry in the spherical harmonic functions, the solution is symmetric across the equator. Therefore, we concentrate on the upper hemisphere of the unit sphere. Even though the initial condition is smooth, we observe that the solution develops shocks as it evolves in time. Two shocks move towards a stationary shock and interact with each other in a non-straightforward manner.

\subsection{Conservation of Mass}
\label{SubSec:ExBurgerConservation}


\begin{figure}[h!]
\centering
\includegraphics[width=\textwidth]{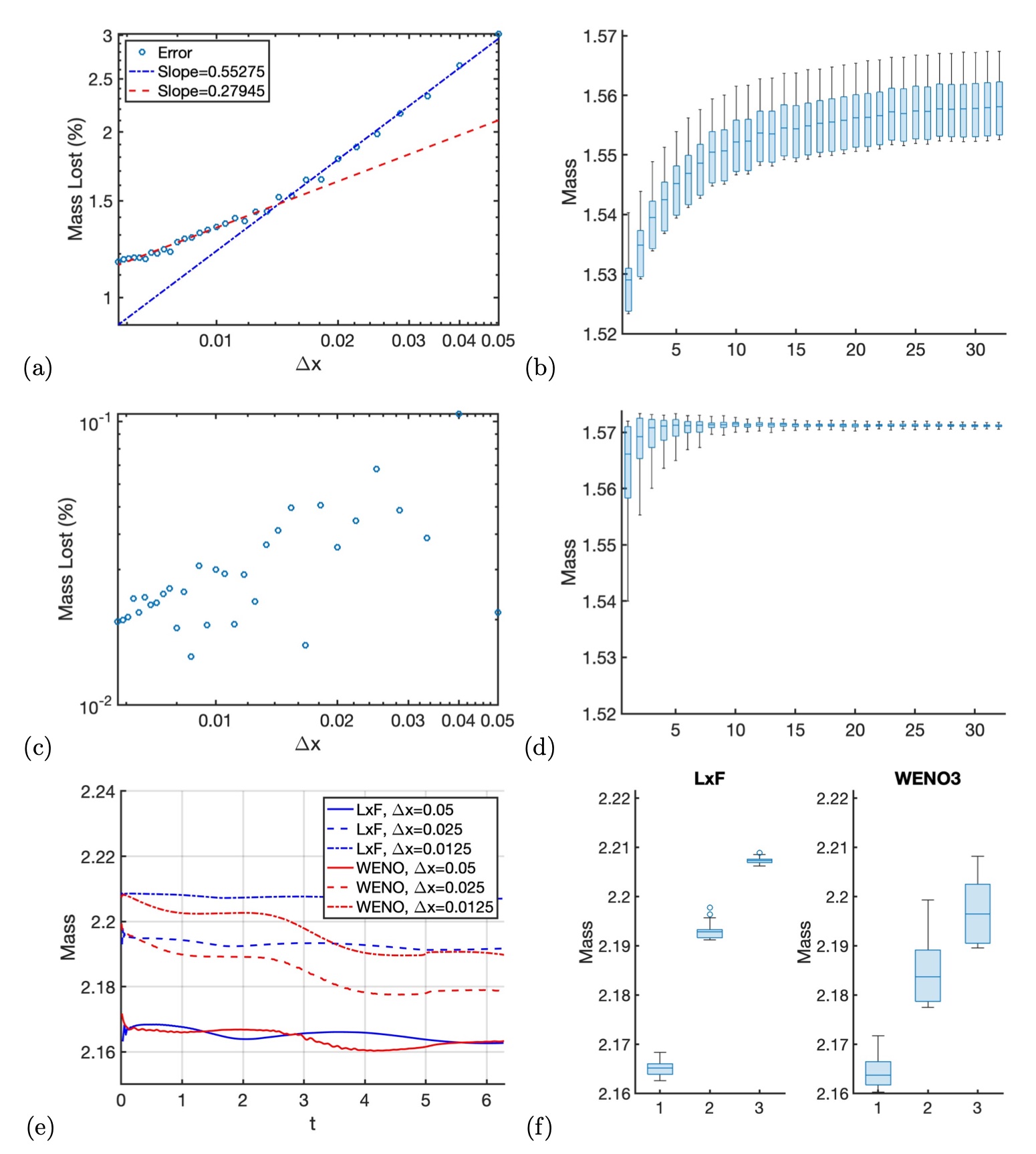}
\caption{
(Section \ref{SubSec:ExBurgerConservation} Conservation of mass in \reminder{Burgers' equation}) 
(a) The integral of $u$ on the unit circle at $t=2\pi$. 
(b) The change in the integral of $u$ on the unit circle. 
\reminder{(c-d) We repeat the setup in (a-b) but with the exact solution assigned in the outer tube.}
(e) The integral of $u$ on the unit sphere as a function of time. 
(f) The change in the integral of $u$ on the unit sphere. 
}
\label{Fig:ConservationMass}
\end{figure}

The conservation of mass is a fundamental property in nonlinear conservation laws. In this section, we aim to investigate the conservation of the solution $u$ in \reminder{Burgers' equation} on the circle and the sphere. Specifically, we are interested in determining how well the numerical solutions preserve the quantity on the interface.

To this end, we consider solving \reminder{Burgers' equation} on the unit circle using WENO3 with the initial condition $u(\theta)=1$ if $|\theta|\le \pi/4$ and $u(\theta)=0$ otherwise, as shown in Figure \ref{Fig:ConservationMass}(a) and (b). We apply high-order interpolation to the numerical solutions at different times on the unit circle and integrate the total mass. Figure \ref{Fig:ConservationMass}(a) shows the percentage of mass lost at the final time $t=2\pi$ with the numerical solutions computed on various meshes, $\Delta x=4/(n-1)$ with $n$ increasing from 81 to 701 in increments of 20. We compare the total mass in the numerical solution with the exact mass $\pi/2$.

The mass lost in WENO3 at the final time reduces as we refine the mesh. There could be two contributions to this error. The first one is the discretization error. The more grid points we use to sample the circle, the less discretization error we have. The second source comes from the flux through the grid interfaces. The mass lost in high-order finite difference schemes is not surprising, especially for problems with non-periodic boundary conditions \cite{dinshuzha20}, since the method requires information from several ghost points outside the computational domain. To see more clearly the variation in total mass, we have presented a box plot of the total mass as a function of time in Figure \ref{Fig:ConservationMass}(b). The $x$-axis represents different mesh sizes \reminder{$\Delta x=4/(n-1)$ with $n$ ranging from 81 ($x$-index equals 1) to 701 ($x$-index equals 32) in increments of 20}, while the $y$-axis shows the median, lower and upper quartiles, and the minimum and maximum values in the total mass on the circle that are not outliers. As we increase the mesh numbers, i.e., move along the $x$-axis, we observe that the average mass over the $ 2\pi$ time increases and approaches the exact mass $\pi/2$. But more interestingly, the range from the lower to the upper quartile seems to stay roughly the same. This observation implies a non-zero net flux on the computational boundary, which aligns with the results in \cite{ber87,dinshuzha20}. For our problem, the boundary of the computational tube does not align with the computational mesh, making the conservation of mass even more challenging. Nevertheless, we find that the change in the total mass is roughly 1-2\%. \reminder{To further demonstrate the major source of the mass lost, we replace the Neumann boundary condition with the exact solution in the outer tube. In Figure \ref{Fig:ConservationMass}(c-d), we show the mass lost at the final time as a function of the mesh size and also the box plot of the total mass as a function of time. When we have the exact solution as the boundary condition, we see a significant improvement in the mass loss problem. The change in the total mass is now reduced to the order of $O(10^{-2})$.}

We also re-examine the box initial condition $u_2$ as in Section \ref{SubSec:ExBurgerSphere} and solve \reminder{Burgers' equation} until $t=2\pi$. The exact total mass is given by $\pi/\sqrt{2}\approx 2.2214$. In Figure \ref{Fig:ConservationMass}(\reminder{e}), we show the change in the total mass with solutions computed using the LxF and the WENO3 schemes on different mesh sizes. Similar to the two-dimensional case, the mass approaches the exact mass as we refine the underlying mesh. Comparing the solutions from the LxF and the WENO3, we observe that LxF can better preserve the total mass since it does not require any ghost grid and does not need much information outside the computational boundary. This difference is more evident in the box plot in Figure \ref{Fig:ConservationMass}(\reminder{f}). The index in the $x$-axis corresponds to the solution computed using $n=81$ ($\Delta x=0.05$), 161 ($\Delta x=0.025$), and 321 ($\Delta x=0.0125$). Compared to the circle case, the change in the total mass seems reasonable for such small mesh points.

\subsection{Burgers' Equation on a Torus}
\label{SubSec:ExBurgerTorus}

\begin{figure}[!ht]
    \centering
    \includegraphics[width=\textwidth]{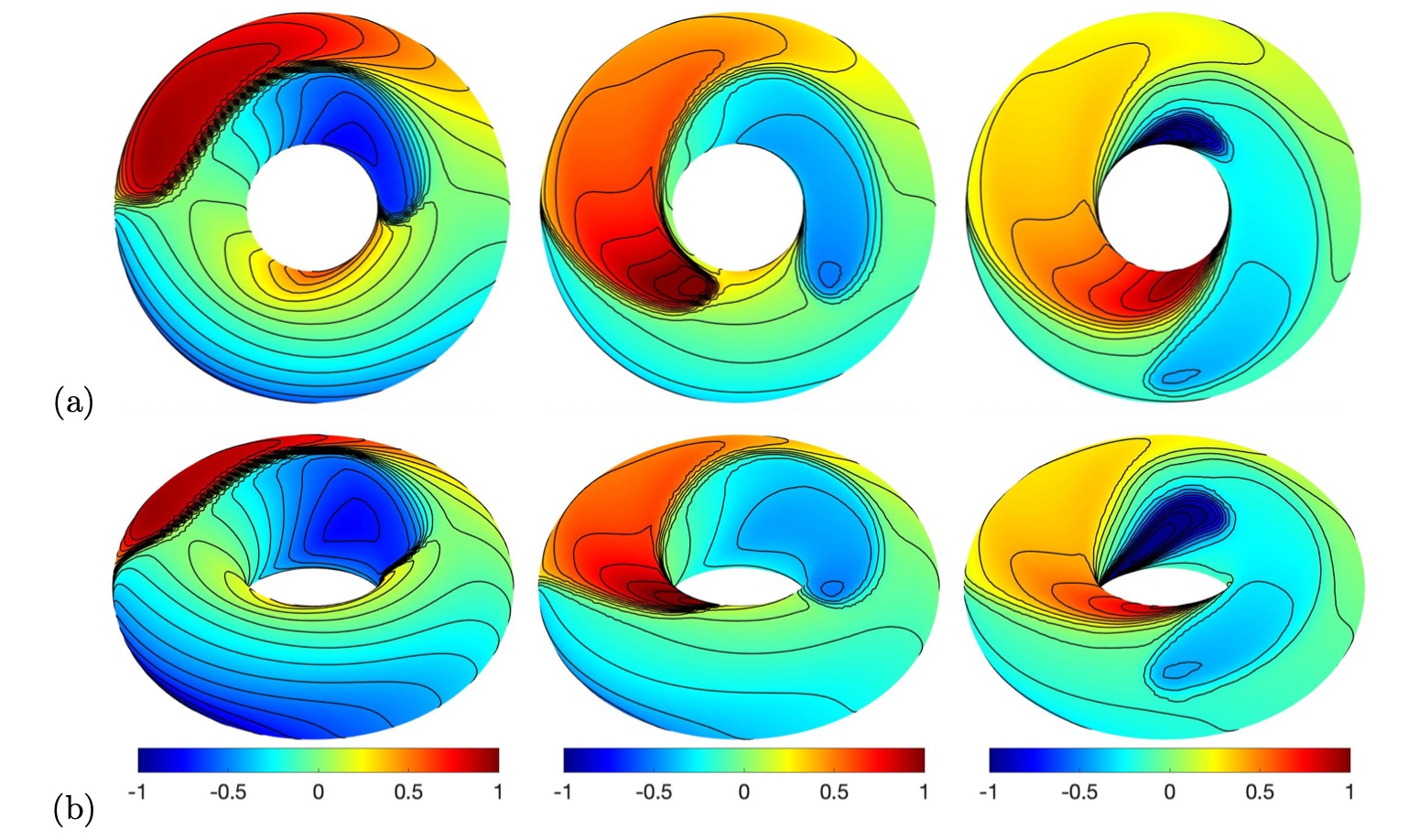}
    \caption{(Section \ref{SubSec:ExBurgerTorus} Burgers' equation on a torus, WENO3) Numerical solutions at $t = 2\pi/5, 6\pi/5$ and $2\pi$ with $\Delta x=0.05$. (a-b) Solutions observed from different angles.}
    \label{Fig:BurgerTorus005}
    \end{figure}

    \begin{figure}[!ht]
    \centering
    \includegraphics[width=\textwidth]{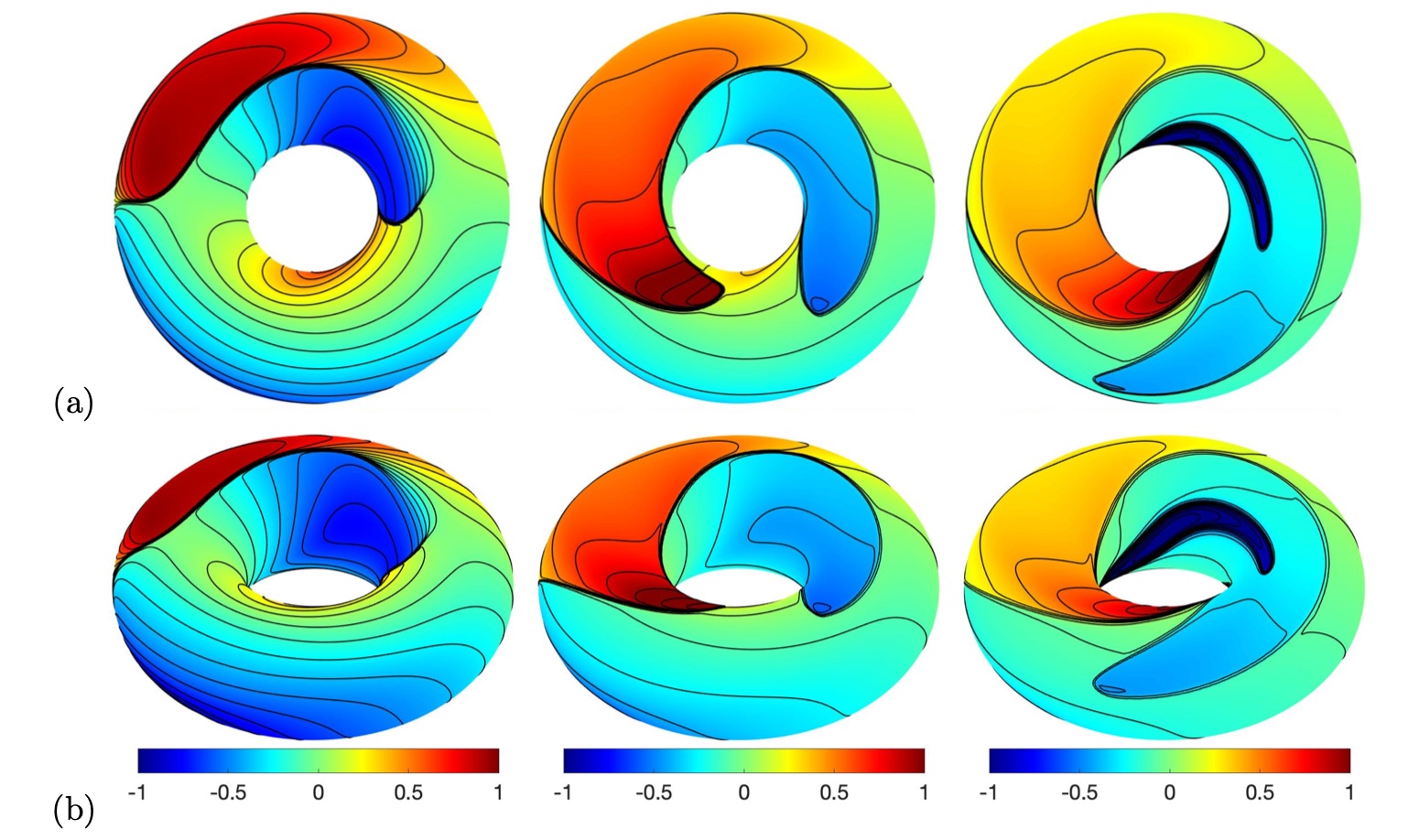}
    \caption{(Section \ref{SubSec:ExBurgerTorus} Burgers' equation on a torus, WENO3) Numerical solutions at $t = 2\pi/5, 6\pi/5$ and $2\pi$ with $\Delta x=0.0125$. (a-b) Solutions observed from different angles.}

    \label{Fig:BurgerTorus00125}
    \end{figure}

Finally, we consider \reminder{Burgers' equation} on the same torus as described in Section \ref{SubSec:ExAdvectionTorus}. The embedding equation in this example is given by
$
u_t + \nabla \cdot \left\{ [I - \phi(\vec{x})H(\vec{x})]^{-1} \vec{F}(u) \right\} = 0
$
where $\vec{F}(u) = \frac{1}{2}u^2(\uvec{\theta} + \uvec{\eta})$ is the flux function constructed along the two principal directions. We consider the smooth initial condition $u_0(\theta, \eta) = \sin\theta , \cos\eta$. Figures \ref{Fig:BurgerTorus005}-\ref{Fig:BurgerTorus00125} show the solution at various times computed by the mesh given by $\Delta x=0.05$ and $\Delta x=0.0125$, respectively. We can see that the resolution in the solution is significantly improved. As we refine the grid, we can much better resolve the shock and clearly see the fine details in the solution.


\section{Summary}
\label{Sec:Conclusion}

By revealing that the closest-point function $P_{\Gamma}$ is a family of diffeomorphisms from $\Gamma_h$ to $\Gamma$, we design a natural embedding flux via the push-forward concept from differential geometry. The proposed method does not rely on local coordinates of $\Gamma$ and performs all the computation on Cartesian grids, so we can utilize any well-developed numerical methods. The numerical examples demonstrate that the solution is constant along the normal directions, and the proposed method can achieve high-order convergence.

There are many possible future directions to explore. One simple possibility is to extend the method to solve systems of conservation laws on surfaces. Secondly, since this approach works perfectly well with the level set method, we propose incorporating the method to other interface problems involving surface gradient or solving PDEs and variational problems on moving surfaces. Also, it is interesting to extend the method to PDEs on non-closed surfaces with a boundary, which poses challenges in incorporating boundary conditions in the embedding framework and coupling constraints with the push-forward and projection operator. \reminder{Parallel implementations will be considered in the future.} Finally, a high-order method to compute surface integrals could be developed following recent approaches in \cite{kublik2016integration,martin2020equivalent}.

\bibliographystyle{plain}
\bibliography{syleung,Bibliography}

\end{document}